\documentclass[11pt,a4paper]{article}

\usepackage[T1]{fontenc}
\usepackage{amsmath}
\usepackage{amsthm}
\usepackage{amsfonts}
\usepackage{subcaption}
\usepackage{float}
\usepackage{algorithm}
\usepackage{algpseudocode}
\usepackage{amssymb}
\usepackage{graphicx}
\usepackage{tcolorbox}
\usepackage{mathtools}
\usepackage{booktabs}
\usepackage{pgfplotstable,array}

\pgfplotsset{compat=1.18}
\pgfplotstableset{
    zerofill,
    precision=3,
    col sep=comma,
    header=true,
    fixed,
    font={\scriptsize},
    row style/.style 2 args={
        every row #1 column 0/.style={#2},
        every row #1 column 1/.style={#2},
        every row #1 column 2/.style={#2},
        every row #1 column 3/.style={#2},
        every row #1 column 4/.style={#2},
        every row #1 column 5/.style={#2},
        every row #1 column 6/.style={#2},
        every row #1 column 7/.style={#2},
        every row #1 column 8/.style={#2},
        every row #1 column 9/.style={#2},
        every row #1 column 10/.style={#2},
    },
}
\usepackage{siunitx}

\usepackage[table,xcdraw]{xcolor}
\usepackage{multirow}
\usepackage{rotating,tabularx}

\usepackage[left=2cm,right=2cm,top=2cm,bottom=2cm]{geometry}
\newcolumntype{L}[1]{>{\raggedright\let\newline\\\arraybackslash\hspace{0pt}}m{#1}}
\newcolumntype{C}[1]{>{\centering\let\newline\\\arraybackslash\hspace{0pt}}m{#1}}
\newcolumntype{R}[1]{>{\raggedleft\let\newline\\\arraybackslash\hspace{0pt}}m{#1}}

\newtheorem{Theorem}{Theorem}[section]
\newtheorem{Proposition}[Theorem]{Proposition}
\newtheorem{Remark}[Theorem]{Remark}
\newtheorem{Lemma}[Theorem]{Lemma}

\newtheorem{Definition}[Theorem]{Definition}
\newtheorem{Example}[Theorem]{Example}

\newtheorem{Step}{Step}
\usepackage{hyperref}
\hypersetup{colorlinks=true,urlcolor=blue}
\expandafter\let\expandafter\oldproof\csname\string\proof\endcsname
\let\oldendproof\endproof
\renewenvironment{proof}[1][\proofname]{
\oldproof[\ttfamily\scshape \bf #1.]
}{\oldendproof}

\def\tilde{\widetilde}

\def\min{\mbox{\rm minimize}}

\def\ox{\overline{x}}

\def\Tilde{\widetilde}
\def\Bar{\overline}

\def\epsilon{\varepsilon}
\def\ox{\bar{x}}

\def\ph{\varphi}

\def\lm{\lambda}

\def \N{{\rm I\!N}}
\def \R{{\rm I\!R}}

\newcommand{\dotproduct}[1]{\left\langle#1\right\rangle}

\newcommand{\norm}[1]{\left\|#1\right\|}

\numberwithin{equation}{section}

\title{\bf New Globalized Newton-Type Methods for Nonconvex Optimization Problems}
\author{Vo Thanh Phat\footnote{Department of Mathematics and Statistics,  University of North Dakota, Grand Forks, North Dakota, USA. E-mail: thanh.vo.1@und.edu.}\quad\quad\quad  Tuyen Tran\footnote{Department of Mathematics  and Statistics,  Loyola University Chicago, Chicago, Illinois, USA. E-mail: ttran18@luc.edu.}}
\begin{document}  
\maketitle
\noindent
{\small{\bf Abstract}. Newton's method is one of the most effective second-order algorithms for smooth optimization because of its fast local convergence. However, existing globally convergent Newton-type methods typically require convexity or strong convexity of the objective function, while approaches for nonconvex optimization often rely on Hessian regularization at every iteration. In this paper, we propose a general line-search Newton framework for unconstrained optimization that avoids repeated Hessian regularization by exploiting the Newton direction only when it is well-defined and suitable. The proposed framework encompasses several existing hybrid gradient--Newton methods as special cases and naturally yields a new extragradient Newton method. We establish global convergence under mild assumptions, including the Polyak--\L ojasiewicz--Kurdyka (PLK) condition, allowing both isolated and nonisolated accumulation points. We further prove local superlinear and quadratic convergence under appropriate regularity assumptions. Finally, we apply the proposed framework to strongly quasiconvex optimization and provide, to the best of our knowledge, the first Newton-type algorithm together with a comprehensive convergence analysis for this important class of nonconvex optimization problems. Numerical experiments demonstrate the effectiveness of the proposed methods.    \\[1ex]
 }\vspace*{-0.1in}

\section{Introduction}
The classical Newton method was originally introduced for solving systems of nonlinear equations (see \cite{ds96,kel03}) and was later adapted to unconstrained optimization problems; see, for example, the monographs \cite{BeckAl,ds96,Solo14}. Consider the unconstrained optimization problem
\begin{equation}\label{opproblem}
\min  \quad \varphi(x) \quad \text{subject to } x\in\R^n,
\end{equation}
where the objective function $\varphi$ is twice continuously differentiable. The Newton iteration is derived from the second-order Taylor approximation of the objective function and generates a sequence $
x^{k+1}=x^k+d^k$, 
where the search direction $d^k$ is obtained by solving the linear system
\begin{equation}\label{classical}
-\nabla\varphi(x^k)=\nabla^2\varphi(x^k)d^k,\qquad k=0,1,2,\ldots.
\end{equation}
Here, $\nabla\varphi(x^k)$ and $\nabla^2\varphi(x^k)$ denote the gradient vector and Hessian matrix of $\varphi$ at $x^k$, respectively. Owing to its local quadratic convergence, the Newton method is one of the most influential second-order methods in nonlinear optimization.

\medskip 
Despite its remarkable local convergence properties, the classical Newton method is not globally convergent in general. Consequently, numerous globalization techniques have been developed to ensure convergence from arbitrary initial points. Among these, line-search strategies constitute one of the most successful approaches and have been extensively investigated under the assumption that the objective function is strongly convex; see, for example, \cite{BeckAl,Boyd,burke}. These globally convergent Newton frameworks have become the cornerstone for many subsequent developments, including  regularized Newton methods \cite{dmn24,gn17,lfqy04,mish23,poly09}, coderivative-based Newton methods for nonsmooth optimization \cite{kmptjogo,kmp,kmptmp,sc24}, and stochastic Newton methods \cite{bbn18,bcnn11,rm19}.

\medskip 
Nevertheless, most existing Newton-type methods require the objective function to be convex, or even strongly convex, to establish global convergence. Such assumptions are often restrictive in modern optimization problems arising from machine learning, economics, signal processing, and data science, where objective functions are frequently nonconvex. This limitation has motivated the development of   regularized Newton methods for nonconvex problems, which incorporate Hessian modifications or regularization terms to guarantee descent directions even in nonconvex settings; see, for example,  \cite{uy10,uy14}. However, these methods usually require modifying or regularizing the Hessian matrix at every iteration, resulting in additional computational costs and algorithmic complexity.

\medskip 
The objective of this paper is to further extend Newton's method to important classes of nonconvex optimization problems without requiring Hessian modifications at every iteration. In particular, we focus on optimization problems whose objective functions are quasiconvex or satisfy the Polyak--\L ojasiewicz--Kurdyka (PLK) condition. These function classes have attracted considerable attention in recent years because they encompass many practically relevant optimization models while allowing significantly weaker assumptions than strong convexity. The key feature of our approach is that the Newton direction is used only when it exists and is suitable, rather than being enforced at every iteration through Hessian regularization as in regularized Newton methods. This leads to a computationally attractive algorithm while maintaining strong convergence guarantees under suitable assumptions. The main contributions of this paper can be summarized as follows.
\begin{itemize}
\item We propose a new Newton framework for general nonconvex optimization problems equipped with a line-search globalization strategy. The proposed framework unifies several existing hybrid gradient--Newton methods as special cases and naturally gives rise to a novel extragradient Newton method.

\item We establish global convergence of the proposed method under mild assumptions. Our analysis covers both isolated and non-isolated accumulation points through the Polyak--\L ojasiewicz--Kurdyka (PLK) condition.

\item We further investigate local convergence properties and prove superlinear or quadratic convergence under appropriate regularity assumptions.

\item We apply the proposed framework to strongly quasiconvex optimization, an important class of optimization problems that has recently received considerable attention in economics and optimization theory. Existing algorithms for this class are almost exclusively first-order methods \cite{glm25,l22,l25}. To the best of our knowledge, this paper provides the first Newton-type algorithm together with a comprehensive convergence analysis for strongly quasiconvex optimization.
\end{itemize}

The paper is organized as follows. Section~\ref{sec:general} introduces a general Newton framework with line-search globalization and establishes preliminary convergence properties. Section~\ref{sec:PLK} investigates global convergence under the PLK condition. Section~\ref{sec:localconvergence} studies local superlinear and quadratic convergence. Section~\ref{sec:quasiconvex} is devoted to the analysis of the proposed method for quasiconvex optimization problems. Numerical experiments illustrating the theoretical results are presented in Section~\ref{sec:num}. Finally, Section~\ref{sec:conclusion} concludes the paper with a brief summary and directions for future research.

\section{General Scheme for   Newton Methods with Line Search}\label{sec:general}
We begin by recalling some basic notions and notation that will be used throughout this paper. All considerations are made in the space $\R^n$ equipped with the Euclidean norm $\|\cdot\|$ and the standard inner product $\left\langle \cdot, \cdot\right\rangle$. We denote by $\N:=\{0,1,2,\ldots\}$ the set of natural numbers, and by $\R_+$ the set of nonnegative real numbers. The symbol $B_r(x)$ signifies  the  open ball  centered at $x$ with radius $r > 0$.   For $J\subset \N$, the notation $x^k \overset{J}{\to} \bar{x}$ indicates that the subsequence $\{x^k\}_{k \in J}$ converges to $\bar{x}$ as $k \to \infty$. For a $\mathcal{C}^1$-smooth function $\varphi:\R^n\to\R$, a point $\bar{x}$ is called a \textit{stationary point} of $\varphi$ if $\nabla \varphi(\bar{x})=0$. For a symmetric matrix $A$, we write $\lambda_{\text{\rm min}}(A)$ to denote its smallest eigenvalue. Recall that a function $\varphi:\R^n \to \R$ is called $\mathcal{C}^{1,1}_L$ for some $L>0$ if $\varphi$ is Fr\'echet differentiable and $\nabla \varphi$ is Lipschitz continuous with modulus $L>0$, i.e.,
$$
\|\nabla\varphi(x) - \nabla\varphi(y) \| \leq L\|x-y\| \quad \text{for all }\;x,y \in \R^n.
$$
It follows from \cite[Lemma A.11]{Solo14}  that when $\varphi$ is $\mathcal{C}^{1,1}_L$ for some  $L>0$, the following holds:
\begin{equation}\label{descentlm}
\varphi(y)\leq \varphi(x) + \langle \nabla \varphi(x), y-x\rangle + \frac{L}{2}\|y-x\|^2 \quad \text{for all } x,y\in\R^n.
\end{equation}
The main idea of our proposed Newton-type method is to replace the classical Newton step \eqref{classical} 
with the modified step
$$
-\nabla\varphi(\hat{x}^k) = \nabla^2\varphi(\hat{x}^k) d^k, \quad k = 0,1,2,\ldots
$$
where $\hat{x}^k$ is selected to be sufficiently close to $x^k$ in an appropriate sense. The illustration comparing our Newton-type method with the classical Newton-type method is shown in Figure~\ref{fig:NM-compare}. As can be seen, instead of constructing the tangent line at $x^k$, we construct it at $\hat{x}^k$. If $\hat{x}^k$ is chosen appropriately, we expect the sequence generated by our method to converge to $x^*$ faster than the classical Newton sequence.  
\begin{figure}
    \centering        \includegraphics[width=\linewidth]{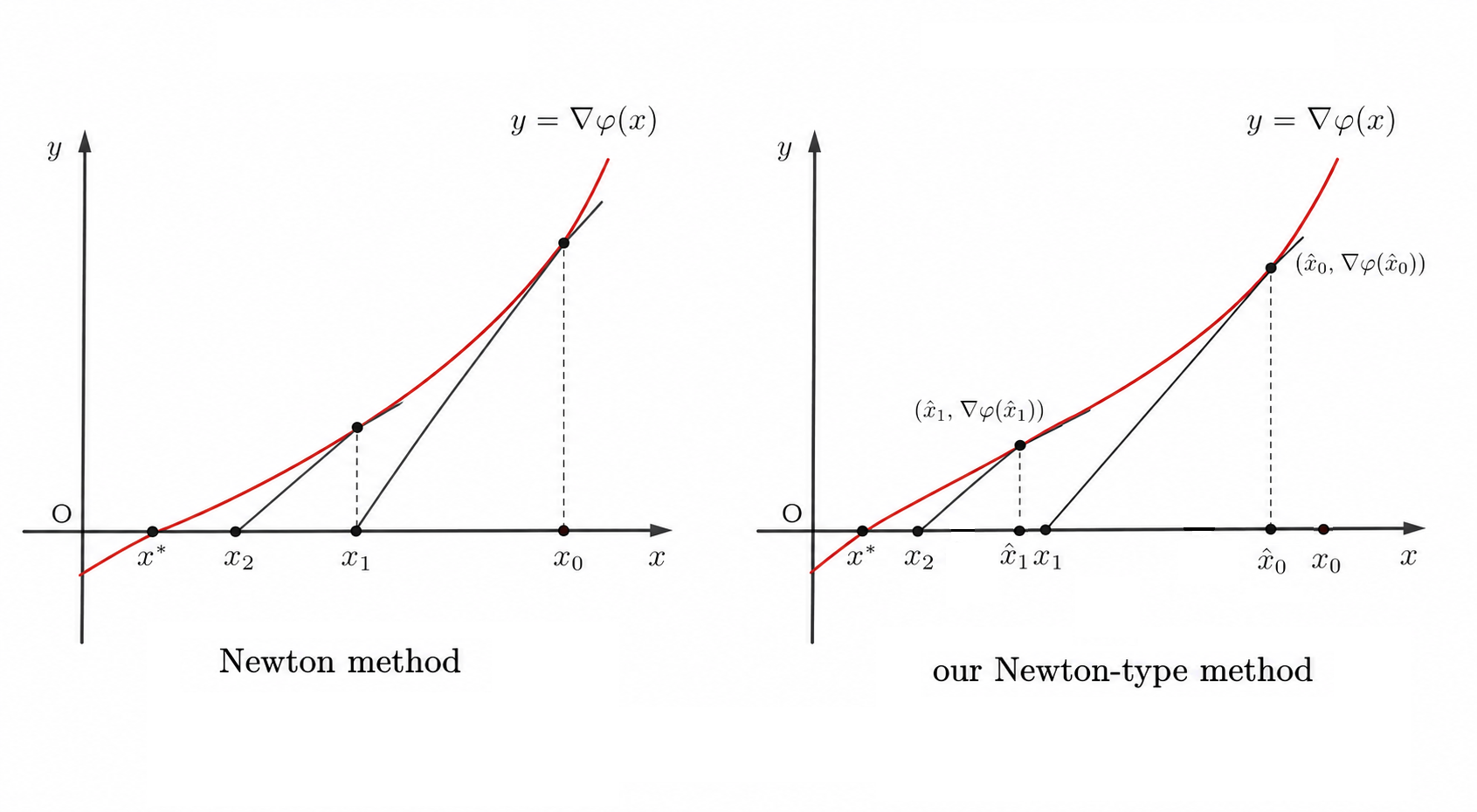}
    \caption{Geometric interpretation of Newton-type iterations}
    \label{fig:NM-compare}
\end{figure}
More precisely, we require that the function value at $\hat{x}^k$ does not exceed that at $x^k$, i.e., $\varphi(\hat{x}^k) \le \varphi(x^k)$. In addition, the distance between $\hat{x}^k$ and $x^k$ must be sufficiently small in an appropriate sense.   
With these requirements, our  globalization scheme for the Newton method is formulated as follows. 

\begin{algorithm}[H]
\caption{(general framework for globalized Newton methods)}\label{globalaNM}
\text{\bf Input:} {$x^0\in\R^n$,  $\sigma \in (0,1/2)$, $\beta\in(0,1)$, $\zeta>0$}.  For $k=1,2,\ldots,$ do the following:
\setcounter{Step}{0}
\begin{Step}\rm ({\bf approximate step}) 
Choose $\hat{x}^k$ satisfying $\varphi(\hat{x}^k)\leq \varphi(x^k)$. If $ \nabla\varphi(\hat{x}^k)=0$, we stop the algorithm. Otherwise, we  go to the next step.
\end{Step} 
\begin{Step}\rm Choose $d^k \in \R^n$ satisfying
$$
d^k = \begin{cases}
-\nabla^2 \varphi(\hat{x}^k)^{-1}\nabla\varphi(\hat{x}^k) & \text{if }\;  \lambda_{\rm min}\left(\nabla^2 \varphi(\hat{x}^k)\right)\geq \zeta \\
-\nabla\varphi(\hat{x}^k) & \text{otherwise}
\end{cases}
$$
\end{Step}
\begin{Step}\rm Set $\tau_k:= \text{\rm argmax} \left\{\tau|\; \varphi(\hat{x}^k +\tau d^k) \leq  \varphi(\hat{x}^k) + \sigma \tau \dotproduct{\nabla\varphi(\hat{x}^k),d^k}   , \tau =\beta^{i}, i \in \N \right\}$.
\end{Step}
\begin{Step}\rm Set $x^{k+1}:= \hat{x}^k + \tau_k d^k.$
\end{Step}
\end{algorithm}

\begin{Remark}[\bf selections of the approximate step in Algorithm \ref{globalaNM}]\label{choicestep} \rm  The choice of $\hat{x}^k$ in Algorithm \ref{globalaNM} gives rise to several variants of Newton-type methods. A natural and simplest choice is to set $\hat{x}^k = x^k$, in which case Algorithm \ref{globalaNM} reduces to the {\em hybrid gradient-Newton} method; see e.g. \cite{BeckAl} and the references therein. 

\medskip 
Another choice, which will be investigated in this paper, is to define $\hat{x}^k$ as
\begin{equation}\label{extragrad}
\hat{x}^k:=x^k - \lm \nabla \varphi(x^k), \quad k =0, 1, 2,\ldots
\end{equation}
where $\lm \in (0,2/L)$ and $L$ is the Lipschitz constant of $\nabla \varphi$. In this case, we refer to the resulting algorithm as the {\em extragradient Newton method}, and we will evaluate its performance through numerical experiments in Section \ref{sec:num}. 
We will show that the choice in \eqref{extragrad} satisfies the requirements of Algorithm \ref{globalaNM} under the assumption that $\varphi$ is $\mathcal{C}^{1,1}_L$ for some  $L>0.$ Indeed, by \eqref{descentlm}, we have 
$$
\varphi(\hat{x}^k) \leq \varphi(x^k) +\dotproduct{\nabla \varphi(x^k), \hat{x}^k-x^k} + \frac{L}{2} \norm{\hat{x}^k-x^k}^2,
$$
or
$$
\varphi(\hat{x}^k) \leq \varphi(x^k) - \lm\left(1-\frac{L}{2}\lm   \right)\norm{\nabla\varphi(x^k)}^2 \leq \varphi(x^k). 
$$
\end{Remark}

We next study the convergence properties of the general framework in Algorithm \ref{globalaNM} under a general choice of approximate steps. To this end, we first establish several auxiliary results.

\begin{Proposition}\label{uniformPD} Let $\varphi\colon \R^n \to \R$ be a $\mathcal{C}^2$-smooth function, and let $\bar{x} \in \R^n$ be such that $\nabla^2\varphi(\bar{x})$ is positive definite. Then the minimum eigenvalue of the Hessian satisfies 
\begin{equation}\label{symmetricmin}
\alpha:=\lambda_{\rm min}(\nabla^2\varphi(\bar{x}))>0 \quad \text{and }\; \dotproduct{\nabla^2 \varphi(\ox)w,w} \geq \alpha \|w\|^2 \quad \text{for all }\; w \in \R^n.
\end{equation}
Furthermore, for every $\kappa \in (0,\alpha)$, there exists $r>0$ such that 
\begin{equation}\label{neighbornorm}
\dotproduct{\nabla^2\varphi(x)w,w}  \ge \kappa \|w\|^2 \quad \text{for all } w \in \R^n, x \in B_r(\ox). 
\end{equation}
\end{Proposition}
\begin{proof} Since $\varphi$ is $\mathcal{C}^2$-smooth, it follows that $\nabla^2 \varphi(\ox)$ is symmetric, and thus \eqref{symmetricmin} holds due to basic properties of Rayleigh's quotient. We only need to verify \eqref{neighbornorm}.   Indeed, we pick any $\kappa \in (0,\alpha)$, then it follows from the continuity of $\nabla^2\varphi$ at $\ox$, we can find $r>0$ such that 
$$
\norm{\nabla^2 \varphi(x) -\nabla^2\varphi(\ox)} \leq  \alpha - \kappa  \quad \text{for all }\; x \in B_r(\ox). 
$$
Therefore, for any $w \in \R^n, x \in B_r(\ox)$, we have the following estimates
\begin{align*}
    \dotproduct{\nabla^2\varphi(x)w,w} & =  \dotproduct{\nabla^2\varphi(\ox)w,w} + \dotproduct{(\nabla^2\varphi(x) - \nabla^2\varphi(\ox))w,w}\\
    & \geq \alpha \|w\|^2 - \norm{\nabla^2\varphi(x) - \nabla^2\varphi(\ox)}\|w\|^2 \\
    & \geq \alpha \|w\|^2 - (\alpha - \kappa) \|w\|^2 = \kappa \|w\|^2,
\end{align*}
which completes the proof. 
   
\end{proof} 

The following lemma establishes the existence of a descent direction, which plays an important role in the convergence analysis of the proposed Newton-type methods in this section.
\begin{Lemma}\label{existNM} Let $\varphi\colon \R^n \to \R$ be a $\mathcal{C}^2$-smooth function, a positive number $\zeta >0$, and let  $x\in\R^n$. Assume that $\nabla\varphi(x)\ne 0$, and choose $d\in\R^n$ according to
\begin{equation}\label{directionNM}
d=
\begin{cases}
\text{the solution of } -\nabla\varphi(x)=\nabla^2\varphi(x)d, & \text{if }\lambda_{\text{\rm min}}\left(\nabla^2\varphi(x)\right)\geq \zeta,\\[4pt]
-\nabla\varphi(x), & \text{otherwise}.
\end{cases}
\end{equation}
Then  we have $d \ne 0$, and 
\begin{equation}\label{xi}
\langle \nabla\varphi(x),d\rangle \leq - \xi \|d\|^2  \quad \text{where }\;  \xi:= \text{\rm min}\{1,\zeta\}>0. 
\end{equation}
Moreover, for every $\sigma\in(0,1)$ there exists $\delta>0$ such that
$$
\varphi(x+\tau d)\le \varphi(x) + \sigma\tau\langle \nabla \varphi(x),d\rangle 
\quad\text{for all }\;\tau\in(0,\delta).
$$
\end{Lemma}
\begin{proof} Let $x$  in $\R^n$, and assume that $\nabla\varphi(x) \ne 0$. Since $d$ is defined as in \eqref{directionNM}, we have $d \ne 0$.  If $\kappa:=  \lambda_{\rm min}(\nabla^2 \varphi(x))\geq \zeta$, it follows from Proposition \ref{uniformPD} that 
$$
\dotproduct{\nabla^2\varphi(x) w, w} \ge \kappa \|w\|^2 \geq \zeta \|w\|^2 \quad \text{for all } w \in \R^n.
$$ 
In this case, the direction $d$ in \eqref{directionNM} satisfies $
\langle \nabla\varphi(x), d \rangle   \leq -\zeta \|d\|^2.$ 
Otherwise, the direction $d$ is chosen as $d = -\nabla\varphi(x)$, which gives $
\langle \nabla\varphi(x), d \rangle = -\|d\|^2$. 
In both cases, we have $\langle \nabla\varphi(x), d \rangle \leq -\xi \|d\|^2 <0 $, where $\xi$ is defined as in \eqref{xi}. 
Hence,
$$
\lim_{t \to 0^+} \frac{\varphi(x+t d)-\varphi(x)}{t} = \langle \nabla\varphi(x), d \rangle < \sigma\langle \nabla\varphi(x), d \rangle 
$$
which implies the existence of $\delta>0$ such that
$$
\varphi(x+ \tau d) \le \varphi(x)  +\sigma \tau \langle \nabla\varphi(x), d \rangle  \quad \text{for all } \tau \in (0,\delta).
$$
The proof is complete.
\end{proof}

The following lemma is a local version of \cite[Lemma 2.20]{Solo14}, restated in \cite[Lemma 3]{kmptmp}.
\begin{Lemma} \label{estimate1} Let  $\Omega \subset \R^n$ be an open set, and let $\varphi:\Omega \to\R$ be a continuously differentiable function such that $\nabla\varphi$ is Lipschitz continuous with modulus $L >0$. Let $x \in \Omega$,  and $d \in \R^n\setminus\{0\}$  satisfy $\langle \nabla \varphi(x),d\rangle <0$. Then for any $\sigma\in (0,1)$, the following inequality 
\begin{equation}\label{est1}
\varphi(x+\tau d)\leq \varphi(x) + \sigma \tau \langle \nabla\varphi(x),d\rangle 
\end{equation}
holds whenever $\tau \in (0,\overline{\tau}]$ and $x+\tau d \in \Omega$, where
\begin{equation}\label{bartau}
\overline{\tau}:= \frac{2(\sigma-1)\langle \nabla\varphi(x),d\rangle}{L\|d\|^2}>0. 
\end{equation}
\end{Lemma}
 
We are now ready to present the main results of this section, including the well-posedness of the proposed algorithm, the characterization of accumulation points, and conditions that ensure the convergence of the generated iterative sequence.
\begin{Theorem}[\bf stationarity of accumulation points for globalized  Newton methods]\label{globalconver} Let $\varphi\colon \R^n \to \R$ be a $\mathcal{C}^2$-smooth function. Then Algorithm~{\rm\ref{globalaNM}} either terminates after finitely many iterations or generates sequences $\{x^k\}$ and $\{\hat{x}^k\}$ such that the sequences of function values $\{\varphi(x^k)\}$ and $\{\varphi(\hat{x}^k)\}$  are monotonically decreasing. Moreover, we have the following assertions:
\begin{itemize}
\item[\bf (i)] If $\inf \varphi(x^k) >-\infty$, then $\displaystyle\sum_{k=1}^\infty\tau_k \norm{d^k}^2 <\infty.$ 
\item[\bf (ii)] If $\varphi$ is $\mathcal{C}^{1,1}_L$ for some $L>0$ then the sequence $\{\tau_k\}$ is bounded away from $0$.  Consequently, the sequences $\left\{\norm{d^k} \right\}$, $\left\{\norm{\nabla\varphi(\hat{x}^k)} \right\}$ converge to $0$ as $k \to \infty$.
\item[\bf (iii)] If the sequence $\{\hat{x}^k\}$ satisfies the condition \begin{equation}\label{approximates}
\norm{\hat{x}^k-x^k} \overset{J}{\to} 0   \quad \text{if } J \subset \N \; \text{and } \; \{x^k\}_{k \in J} \; \text{is convergent}. 
\end{equation}
Then every accumulation point of $\{x^k\}$ is a stationary point of $\varphi$. If in addition that $\ox$ is  an isolated accumulation point, then $\{x^k\}$ converges to $\ox$ as $k\to \infty.$  
\end{itemize}
\end{Theorem}
\begin{proof}  Lemma~\ref{existNM}  easily ensures by induction that Algorithm~\ref{globalaNM} either stops after finitely many iterations, or produces   sequences $\{x^k\}$ and $ \{\hat{x}^k\}$ such that $\varphi(x^{k+1})<\varphi(x^k)$ and $\varphi(\hat{x}^{k+1})<\varphi(\hat{x}^k)$  for all $k\in\N$. It follows from Lemma \ref{existNM} that 
\begin{equation}\label{boundeddirection}
\dotproduct{\nabla\varphi(\hat{x}^k),d^k}  \leq -\xi \norm{d^k}^2 \quad \text{for all }\; k \in \N, \; \text{where } \xi:=\text{\rm min}\{1,\zeta\}>0.
\end{equation}
We first verify the assertion {\bf (i)}. Indeed, by \eqref{boundeddirection} and the construction of $\{x^k\}$ in Algorithm \ref{globalaNM}, we have 
\begin{equation}\label{sumconverges}
\tau_k \|d^k\|^2 \leq \frac{1}{\xi}\tau_k \langle -\nabla\varphi(\hat{x}^k),d^k\rangle \leq \frac{1}{\xi\sigma}\left(\varphi(\hat{x}^k)-\varphi(x^{k+1}) \right)\leq \frac{1}{\xi\sigma}\left(\varphi(x^k)-\varphi(x^{k+1}) \right)
\end{equation}
Since the sequence $\{\varphi(x^k)\}$ is nonincreasing, and $\inf  \varphi(x^k) >-\infty$, the sequence $\{\varphi(x^k)\}$ must converge to a number $\varphi^*$ as $k\to\infty$. Combining the latter with \eqref{sumconverges}, we  have  
$$
\sum_{k=1}^\infty \tau_k \norm{d^k}^2 \leq \frac{1}{\xi\sigma}\left(\varphi(x^0) -\varphi^*\right) <\infty,
$$
which justifies {\bf (i)}. To verify {\bf (ii)}, we assume further that $\varphi$ is $\mathcal{C}^{1,1}_L$ for some $L>0.$ Applying Lemma~\ref{estimate1} and using \eqref{boundeddirection}, we have 
$$
\beta^{-1}\tau_{k} > \frac{2(\sigma -1)\dotproduct{\nabla \varphi(\hat{x}^{k}),d^{k}}}{L\norm{d^{k}}^2} \geq \frac{2(1-\sigma)\xi}{L}>0,
$$
which justifies the boundedness of $\{\tau_k\}$ away from $0$. Combining this with {\bf (i)}, we deduce that the sequence $\left\{\norm{d^k} \right\}$ converges to $0$ as $k \to \infty.$ Moreover, since $\varphi$ is $\mathcal{C}^{1,1}_L$, we have 
$$
\norm{\nabla^2\varphi(x)w}\leq L\|w\| \quad \text{for all }\; w \in \R^n,
$$
which implies that 
$$
\norm{\nabla\varphi(\hat{x}^k)} \leq \text{\rm max} \left\{ \norm{d^k}, \norm{\nabla^2\varphi(\hat{x}^k)d^k} \right\} \leq \text{\rm max} \left\{ \norm{d^k}, L\norm{d^k} \right\} = \text{\rm max}  \left\{1,L \right\} \norm{d^k}
$$
for all $k \in \N.$ Combining the latter with the fact that $d^k \to 0$ as $k \to \infty$, we obtain the convergence of $\left\{\norm{\nabla\varphi(\hat{x}^k)} \right\}$ to $0$ as $k \to \infty.$ 

\medskip 
It remains to verify {\bf (iii)}. To proceed, let $\ox$ be an accumulation point of the sequence $\{x^k\}$.
Since $\ph$ is of class ${\cal C}^{2}$ around $\bar{x}$, we can find $\delta>0$ such  that $\nabla\varphi$ is Lipschitz continuous on $B_\delta(\bar{x})$ with some constant $\ell>0$, which implies that 
\begin{equation}\label{uniformLip}
\norm{\nabla^2\varphi(x)w}\leq\ell\|w\|\quad\text{for all }\; x\in B_\delta(\bar{x}),\;\mbox{ and }\;w\in\R^n.
\end{equation}

The rest of the proof is split into the following  claims.

\medskip \noindent
{\bf Claim~1:} {\em For any subset $J\subset \N$  such that $x^{k} \overset{J}{\to}\bar{x}$, the corresponding sequence $\{\tau_{k}\}_{k \in J}$ in Algorithm~{\rm\ref{globalaNM}} is bounded from below by some $\gamma>0$, the corresponding sequence $\{d^{k}\}_{k \in J}$ is bounded, and we have
\begin{equation}\label{dkLips}
\norm{\nabla \varphi(\hat{x}^{k})} \leq \alpha \norm{d^{k}} \quad \text{where }\; \alpha := \text{\rm max}\{1,\ell\},
\end{equation}
\begin{equation}\label{xk+1}
\varphi(x^{k})-\varphi(x^{k+1}) \ge\sigma\gamma\xi\norm{d^{k}}^2,
\end{equation}
for all large $k\in J$}. 

\medskip Since  $x^{k} \overset{J}{\to} \ox$  and \eqref{approximates} holds, we have $\hat{x}^{k} \overset{J}{\to} \ox$. Combining this with \eqref{uniformLip},  the construction of the sequence $\{\hat{x}^k\}$ implies \eqref{dkLips}.  
Moreover, it follows from \eqref{boundeddirection} that the Cauchy-Schwarz inequality yields $\norm{\nabla\varphi(\hat{x}^{k})}\ge\xi\norm{d^{k}}$ for all $k \in \N$. Since $\hat{x}^{k} \overset{J}{\to}\bar{x}$, the latter estimate verifies the boundedness of the sequence of directions $\{d^{k}\}_{k \in J}$. It remains to show that  $\{\tau_{k}\}_{k\in J}$ is bounded from below by a positive number. Indeed, supposing on the contrary that the opposite holds and combining this with $\tau_k\ge 0$ give us a subsequence of $\{\tau_{k}\}_{k\in J}$ that converges to $0$. Assume without loss of generality that $\tau_{k} \overset{J}{\to} 0$.
 Thus $\hat{x}^{k}+\tau_{k}d^{k} \overset{J}{\to} \bar{x}$, and hence $
\hat{x}^{k}+\tau_{k}d^{k}\in B_\delta(\bar{x})$ whenever $k \in J$ is sufficiently large. Applying Lemma~\ref{estimate1}, we have 
$$
\beta^{-1}\tau_{k} > \frac{2(\sigma -1)\dotproduct{\nabla \varphi(\hat{x}^{k}),d^{k}}}{\ell\norm{d^{k}}^2} \geq \frac{2(1-\sigma)\xi}{\ell},
$$
where the second inequality follows from \eqref{boundeddirection}.  
The fact that $\tau_{k} \overset{J}{\to} 0$ gives us $\sigma\ge 1$, a contradiction due to the choice of $\sigma$. Hence there exists $\gamma>0$ such that $\tau_{k}\ge\gamma$ for all $k \in J$. Moreover, using the estimate in \eqref{boundeddirection} allows us to obtain the following 
\begin{equation}\label{ineq}
\varphi(x^{k})-\varphi(x^{k+1})\ge \varphi(\hat{x}^k)-\varphi(x^{k+1})\geq\sigma\tau_{k}\dotproduct{-\nabla\varphi(\hat{x}^{k}),d^{k}}\ge\sigma\gamma\xi\norm{d^{k}}^2 \quad \text{for all }\; k \in J,
\end{equation}
which therefore justifies Claim 1.

\medskip \noindent
{\bf Claim~2:} {\em $\ox$ is a stationary point of $\varphi$.} 

\medskip Indeed, since $\bar{x}$ is an accumulation point of $\{x^k\}$, there exists a subset $J \subset \N$ such that the subsequence $\{x^{k}\}_{k \in J}$ converges to $\bar{x}$, i.e., $x^{k} \overset{J}{\to} \bar{x}$.   As in the proof of  Claim~1,  we have $\hat{x}^k \overset{J}{\to} \ox$, and we can find $\gamma > 0$ such that \eqref{xk+1} holds. Since the sequence $\{\varphi(x^k)\}$ is nonincreasing and since $\varphi(\bar{x})$ is {an accumulation point} of $\{\varphi(x^k)\}$, the sequence $\{\varphi(x^k)\}$ must converge to $\varphi(\bar{x})$ as $k\to\infty$. Combining the latter with the inequality \eqref{xk+1}, we have $\norm{d^{k}} \overset{J}{\to} 0$.  Therefore, \eqref{dkLips} tells us that $\norm{\nabla\varphi(\hat{x}^k)} \overset{J}{\to} 0$, which implies that   $\nabla\varphi(\bar{x})=0$, which readily justifies Claim~2.

\medskip \noindent
{\bf Claim~3:} {\em $x^k \to \ox$ as $k \to \infty$ if $\ox$ is an isolated accumulation point.} 

\medskip To verify this, we use Ostrowski's condition from \cite[Proposition~8.3.10]{JPang}. Supposing next that $\{x^{k}\}_{k\in J}$ is an arbitrary subsequence of $\{x^k\}$ with $x^{k}\overset{J}{\to}\bar{x}$ for some $J\subset \N$, we need to check that
\begin{equation}\label{Ostrowski}
\lim_{ k (\in J)  \to \infty}\norm{x^{k+1}- x^{k}} =0.
\end{equation}
Indeed, Claim~1 gives us $\gamma>0$ such that \eqref{xk+1} holds, which implies in turn that
\begin{align*}
\norm{x^{k+1}- x^{k}} & \leq \norm{x^{k+1}-\hat{x}^k} + \norm{\hat{x}^k-x^k} \\    
& = \tau_k \norm{d^k} + \norm{\hat{x}^k -x^k}\\
& \leq\norm{d^k} +  \norm{\hat{x}^k -x^k}\\
& \leq  \sqrt{\frac{1}{\sigma\gamma\xi}\left(\varphi(x^{k})-\varphi(x^{k+1}) \right)} + \norm{\hat{x}^k -x^k}
\end{align*}
Since \eqref{approximates} holds, we must have 
$$
\lim_{ k (\in J)  \to \infty} \left(\sqrt{\frac{1}{\sigma\gamma\xi}\left(\varphi(x^{k})-\varphi(x^{k+1}) \right)} + \norm{\hat{x}^k -x^k}\right) = 0,
$$
and hence verifies \eqref{Ostrowski}. Finally, it follows from \cite[Proposition~8.3.10]{JPang} that the sequence $\{x^k\}$ converges to $\bar{x}$ as $k\to\infty$, which therefore completes the proof of Claim~3. The proof is complete. 
    
\end{proof}

\begin{Remark}[\bf on the condition of approximate step]\rm As in Theorem \ref{globalconver}, condition \eqref{approximates} is sufficient to ensure the convergence of the iterative sequence $\{x^k\}$ generated by Algorithm \ref{globalaNM}. A natural question is how to choose $\hat{x}^k$ so that this condition is satisfied. We will show that the two choices described in Remark \ref{choicestep} both meet this requirement. Observe the following:
\begin{itemize}
   \item[\bf (i)] A convenient choice is $\hat{x}^k = x^k$, in which case Algorithm \ref{globalaNM} reduces to the {\em hybrid gradient-Newton method} discussed in Remark \ref{choicestep}. In this setting, \eqref{approximates} is automatically satisfied.

\item[\bf (ii)] If $\varphi$ is $\mathcal{C}^{1,1}_L$ for some $L>0$, another option is to set $\hat{x}^k = x^k - \lm \nabla \varphi(x^k)$, where $\lm \in (0,1/L)$. In this case, Algorithm \ref{globalaNM} reduces to the {\em extragradient Newton method} discussed in Remark \ref{choicestep}. We now show that \eqref{approximates} is satisfied in this case. Indeed, for any $k \in \N$, we have the following estimates:
\begin{align*}
    \norm{\hat{x}^k - x^k} = \lm \norm{\nabla\varphi(x^k)} & \leq \lm \left(\norm{\nabla\varphi(x^k) - \nabla\varphi(\hat{x}^k)} + \norm{\nabla\varphi(\hat{x}^k)} \right)\\
    & \leq \lm \left(L\norm{x^k -\hat{x}^k} + \norm{\nabla\varphi(\hat{x}^k)} \right),
\end{align*}
 which implies that 
 $$
 \norm{\hat{x}^k -x^k} \leq \frac{\lm}{1-\lm L} \norm{\nabla\varphi(\hat{x}^k)} \quad \text{for all }\; k \in \N. 
 $$
Combining this with the assertion {\bf (ii)} of Theorem \ref{globalconver}, we deduce that $\norm{\hat{x}^k - x^k} \to 0$ as $k \to \infty$, which tells us that \eqref{approximates} is satisfied. 

\end{itemize}
\end{Remark}

\section{Convergence under Polyak-\L ojasiewicz-Kurdyka Conditions}\label{sec:PLK}
This section is devoted to providing a convergence analysis of our Newton-type methods without requiring the accumulation point to be isolated, as assumed in Theorem \ref{globalconver}. Our analysis relies on a {\em Polyak--\L ojasiewicz--Kurdyka type condition}, which serves as a fundamental assumption for establishing the convergence of several well-known numerical methods, including inexact proximal point methods, inexact proximal gradient methods, and the boosted DC algorithm; see, for example, \cite{bento,kmptjogo25}.

\medskip For functions $\ph$ of class $\mathcal{C}^{1,1}$ (i.e., continuously differentiable with Lipschitz continuous gradients), Polyak introduced the condition
\begin{equation*}
\norm{\nabla\ph(x)}\ge \frac{1}{2M} |\ph(x)-\ph(\bar{x})|^{1/2}, \quad M>0
\end{equation*}
in his paper \cite{polyak1963gradient} and used it to establish linear convergence of gradient descent in Hilbert spaces. Independently, \L ojasiewicz \cite{Lojasiewicz1963} derived the inequality
\begin{equation}\label{pl}
\norm{\nabla\ph(x)}\geq b|\ph(x)-\ph(\bar{x})|^q,\quad b:=1/M(1-q),\quad q\in[0,1),
\end{equation}
for real analytic functions in a finite-dimensional setting. Inequality \eqref{pl} is commonly known as the {\em Polyak-\L ojasiewicz} (PL) condition; see, e.g., \cite{karimi2016linear}. It was later extended by Kurdyka \cite{Kurdyka1998} to definable functions, and further generalized to nonsmooth settings by Bolt\'e et al.\ \cite{bolte2007lojasiewicz} as the {\em Kurdyka-\L ojasiewicz} (KL) inequality. We adopt the term {\em Polyak-\L ojasiewicz-Kurdyka} (PLK) conditions for this class of properties.  For the purposes and scope of this paper, we recall the PLK condition under the assumption that the function $\varphi$ is smooth, for the convenience of the reader. The definition is given below. 

\begin{Definition}[\bf PLK conditions]\label{def:kl} \rm  Let $\ph:\R^n \to\R$ be a $\mathcal{C}^1$-smooth function. We say that the function $\ph$ satisfies:

\begin{itemize}
    \item[\bf (i)] The {\em Polyak-\L ojasiewicz-Kurdyka} (PLK) {\em condition} at $\bar{x}\in{\rm dom}\,\ph$ if there exist a number $\eta \in (0,\infty)$, a neighborhood $U$ of $\bar{x}$, and a concave continuous function $\phi: [0,\eta] \to[0,\infty)$, called the {\em desingularizing function}, such that
\begin{equation}\label{kl0}
\phi(0)=0,\quad\phi\in \mathcal{C}^1(0,\eta),\quad\phi'(s)>0\;\mbox{ for all }\; s\in(0,\eta),\;\mbox{ and }\;
\end{equation}
\begin{equation}\label{desigualdadeklND}
\phi'\big(\ph(x)-\ph(\bar{x})\big)\norm{\nabla\varphi(x)} \ge 1\;\mbox{ for all }\;x\in U\cap[\ph (\bar x) <\ph(x)<\ph(\Bar{x})+\eta]
\end{equation}
\item[\bf (ii)]  The {\em exponent} version of the {\em PLK conditions}  if the desingularizing function in \eqref{kl0}  and \eqref{desigualdadeklND} is selected in the form $\phi(t)=M t^{1-q}$, where $M$ is a positive constant, and where $q\in[0,1)$. We refer to the case where $q\in(0,1/2)$ as the {\em PLK conditions with lower exponents}.
\end{itemize}
\end{Definition}

We next recall the convergence analysis of a general class of descent methods from \cite{bento}. Let the sequence $\{x^k\}$ be generated by a generic algorithm, and assume that the differentiable function $\varphi$ satisfies the following properties:
\begin{itemize}
\item[\textbf{($\mathcal{H}1$)}]\textit{Sufficient decrease}: for each $k \in \N$, we have
\begin{equation}\label{desigualdade1kNDH}
\ph(x^{k+1}) + a \norm{x^{k+1} - x^k}^2 \le\ph(x^k).
\end{equation}
\item[\textbf{($\mathcal{H}2$})] \textit{Relative error}: for each $k\in \N$, there exists
\begin{equation}\label{fran}
\norm{\nabla \varphi(x^k) } \leq b \norm{x^{k+1} - x^k},
\end{equation}
\end{itemize}
with $a,b>0$. The following results are taken from \cite[Theorems 1 and 4]{bento}.

\begin{Lemma}[\bf convergence analysis under PLK condition]\label{PLKlemma} Let $\ph :\R^n \to\R$ be a $\mathcal{C}^1$-smooth function bounded from below, and let the sequence $\{x^k\}$ be constructed by the generic algorithms satisfying either $(\mathcal{H}1)$ and $(\mathcal{H}2)$  properties. If the   PLK condition holds at some accumulation point $\bar{x}\in \R^n$ of $\{x^k\}$, then   $\{x^k\}$ converges to $\bar{x}$ as $k\to\infty$. Moreover, $\bar{x}$ is  a stationary point of $\ph$. If in addition that the exponent PLK property of $\ph$ holds at $\bar{x}$ with $\varphi(t)=Mt^{1-q}$ for some $M > 0$ and $q\in [0,1)$, then the following convergence rates are guaranteed: 
\begin{itemize}
\item[\bf(i)] If $q\in [0, \frac{1}{2})$, then $\{x^k\}$ and $\{\ph(x^k)\}$ are terminated in a finite number of steps to $\bar x$ and $\bar{\ph}:=\ph(\bar x)$, respectively.

\item[\bf(ii)] If $q=\frac{1}{2}$, then the sequences $\{x^k\}$ and $\{\ph(x^k)\}$ converge linearly at $\bar{x}$ and $\bar{\ph}$, respectively.

\item[\bf(iii)] If $\displaystyle q \in(1/2,1)$, then there exists a positive constant $\sigma$ such that
$$
\left\|x^k - \bar{x} \right\| \leq \sigma k^{-\frac{1-q}{2q-1}} \quad \text{for all large }\;k\in\N.
$$
\end{itemize}
\end{Lemma}

Now we are ready to derive our convergence analysis for Algorithm \ref{globalaNM} under the PLK condition. 

\begin{Theorem}[\bf convergence of Newton-type methods under PLK condition] \label{converPLK} In the setting of Theorem~\ref{globalconver}, assume that $\inf \varphi(x^k) >-\infty$, $\nabla \varphi$ is Lipschitz continuous with modulus $L>0$,  and  there exists $c>0$ such that for every $k \in \N$, we have  
\begin{equation}\label{newconditionforhatx}
\dotproduct{\nabla \varphi(x^k), \hat{x}^k-x^k} \leq -c \norm{x^k-\hat{x}^k}^2,
\end{equation}
and 
\begin{equation}\label{constant}
c > \frac{\alpha L}{2\beta \xi} +L,\;  0 < \sigma < 1 - \frac{\alpha L}{2\beta \xi (c-L)}, \; \text{where } \xi:=\text{\rm min}\{1,\zeta\}, \; \alpha:= \max\{1,L\}.  
\end{equation} 
Let $\bar{x}$ be an accumulation point of the sequence $\{x^k\}$ in which the PLK condition holds. Then the entire sequence $\{x^k\}$ converges to $\bar{x}$, which is a stationary point of $\varphi$. If furthermore $\ph$ enjoys the exponent PLK property with $\varphi(t)=M t^{1-q}$ for some $M>0$ and $q\in[0,1)$, then the convergence rates of $x^k\to\bar x$ are as given in Lemma \ref{PLKlemma}. 
\end{Theorem}

\begin{proof} Note that Condition \eqref{constant} implies that $c>L$, and thus $c- L>0$.  Our goal is to apply Lemma \ref{PLKlemma} to establish convergence and derive the convergence rate for the sequence $\{x^k\}$. To this end, we need to verify that $\{x^k\}$ satisfies the properties $(\mathcal{H}1)$ and $(\mathcal{H}2)$. 
It follows from the construction of $\{x^k\}$ in Algorithm \ref{globalaNM} and Lemma \ref{existNM} that for any $k \in \N$, we have 
$$
\dotproduct{\nabla\varphi(\hat{x}^k),d^k}  \leq -\xi \norm{d^k}^2 \quad \text{for all }\; k \in \N,
$$
which implies that 
$$
\varphi(x^{k+1}) \leq  \varphi(\hat{x}^k)  + \sigma \tau_k \dotproduct{\nabla \varphi(\hat{x}^k),d^k} \leq \varphi(x^k) - \sigma \xi \tau_k \norm{d^k}^2
$$
for any $k \in \N$. Therefore, we have the estimates 
\begin{equation}\label{triangle1}
\norm{\hat{x}^k -x^{k+1}}^2 = \tau_k^2 \norm{d^k}^2 \leq   \tau_k \norm{d^k}^2 \leq \frac{1}{\sigma\xi}\left(\ph(x^k)-\ph(x^{k+1}) \right)   
\end{equation}
for all $k \in \N.$ Furthermore, since $\varphi$ is $\mathcal{C}^{1,1}_L$, we deduce from    Lemma \ref{estimate1} that
\begin{equation}\label{taukoffzero}
\beta^{-1}\tau_{k} > \frac{2(\sigma -1)\dotproduct{\nabla \varphi(\hat{x}^{k}),d^{k}}}{L\norm{d^{k}}^2} \geq \frac{2(1-\sigma)\xi}{L} \quad \text{for all }\; k \in \N,  
\end{equation}
and 
$$
\varphi(\hat{x}^k) \leq \varphi(x^k) + \dotproduct{\nabla\ph(x^k),\hat{x}^k-x^k} + \frac{L}{2}\norm{x^k -\hat{x}^k}^2 \quad \text{for all }\; k \in \N.
$$
Combining the latter with \eqref{newconditionforhatx}, we have
$$
\varphi(\hat{x}^k) \leq \varphi(x^k)  - \left(c - \frac{L}{2} \right)\norm{x^k -\hat{x}^k }^2 \quad \text{for all }\; k \in \N,
$$
which implies that 
\begin{equation}\label{triangle2}
\norm{x^k-\hat{x}^k}^2 \leq \frac{2}{2c-L} \left(\ph(x^k) -\ph(\hat{x}^{k}) \right) \leq  \frac{2}{2c-L} \left(\ph(x^k) -\ph(x^{k+1}) \right) \quad \text{for all }\; k \in \N. 
\end{equation}
Combining \eqref{triangle1} and \eqref{triangle2}, we deduce that 
$$
\norm{x^{k+1}-x^k}^2 \leq 2\left(\norm{x^{k+1}-\hat{x}^k}^2 + \norm{x^k -\hat{x}^k}^2  \right) \leq \frac{1}{a}\left(\ph(x^k) -\ph(x^{k+1}) \right) \quad \text{for all }\; k \in \N,
$$
where $a:= \left(\frac{2}{\sigma\xi} + \frac{4}{2c-L} \right)^{-1}>0,$ and thus $\{x^k\}$ satisfies property $(\mathcal{H}1)$ in \eqref{desigualdade1kNDH}. We next verify that $\{x^k\}$ satisfies property $(\mathcal{H}2)$ in \eqref{fran}. Indeed, for any $k \in \N$, we have 
\begin{align}\label{H2prove1}
\norm{\nabla\ph(x^k)} & \leq \norm{\nabla \ph(x^k)- \nabla \ph(\hat{x}^k)}+ \norm{\nabla\ph(\hat{x}^k)} \leq L\norm{x^k -\hat{x}^k} + \norm{\nabla\ph(\hat{x}^k)} 
\end{align}
Furthermore, it follows from \eqref{newconditionforhatx} and the Cauchy-Schwarz's inequality, we deduce that 
$$
-\norm{\nabla\ph(x^k)}. \norm{\hat{x}^k -x^k} \leq \dotproduct{\nabla \varphi(x^k), \hat{x}^k-x^k} \leq -c \norm{x^k-\hat{x}^k}^2,
$$
which implies that 
\begin{equation}\label{H2prove2}
\norm{x^k -\hat{x}^k} \leq \frac{1}{c}\norm{\nabla\ph(x^k)} \quad \text{for all }\; k \in \N. 
\end{equation}
Combining \eqref{H2prove1} and \eqref{H2prove2}, we have 
\begin{equation}\label{gradxxhat}
 \norm{\nabla\ph(x^k)} \leq \frac{c}{c-L}\norm{\nabla \ph(\hat{x}^k)} \quad \text{for all }\; k \in \N. 
\end{equation}
Moreover, since $\varphi$ is $\mathcal{C}^{1,1}_L$, we have 
$$
\norm{\nabla^2\varphi(x)w}\leq L\|w\| \quad \text{for all }\; w \in \R^n,
$$
which implies that 
\begin{equation}\label{H2prove3}
\norm{\nabla\varphi(\hat{x}^k)} \leq \text{\rm max} \left\{ \norm{d^k}, \norm{\nabla^2\varphi(\hat{x}^k)d^k} \right\} \leq \text{\rm max} \left\{ \norm{d^k}, L\norm{d^k} \right\} = \alpha \norm{d^k}
\end{equation}
for all $k \in \N.$ Combining \eqref{taukoffzero}, \eqref{H2prove2}, \eqref{gradxxhat} and \eqref{H2prove3}, we deduce that 
\begin{align*}
\norm{\nabla\ph(x^k)} \leq \frac{c \alpha}{c-L}\norm{d^k} &=\frac{c\alpha}{(c-L)\tau_k} \norm{x^{k+1} -  \hat{x}^k}   \\
& \leq \frac{c\alpha L}{2\beta(1-\sigma)\xi(c-L)}\norm{x^{k+1} -  \hat{x}^k} \\
& \leq \frac{c\alpha L}{2\beta(1-\sigma)\xi(c-L)}\norm{x^{k+1}-x^k} + \frac{c\alpha L}{2\beta(1-\sigma)\xi(c-L)} \norm{x^k -\hat{x}^k}\\
& \leq \frac{c\alpha L}{2\beta(1-\sigma)\xi(c-L)}\norm{x^{k+1}-x^k} + \frac{\alpha L}{2\beta(1-\sigma)\xi(c-L)} \norm{\nabla\varphi(x^k)}
\end{align*}
for all $k \in \N$, which implies that \eqref{fran} holds with 
$$
b:= \left( 1 -  \frac{\alpha L}{2\beta (1-\sigma)\xi (c-L)}\right)^{-1} \cdot \frac{c\alpha L}{2\beta(1-\sigma)\xi(c-L)} >0
$$
due to \eqref{constant}. Therefore, $\{x^k\}$ satisfies properties  $(\mathcal{H}1)$ and $(\mathcal{H}2)$. By applying Lemma \ref{PLKlemma}, the proof is complete. 
 
\end{proof}

\begin{Remark}[\bf on the conditions in Theorem \ref{converPLK}] \rm Conditions \eqref{newconditionforhatx} and \eqref{constant} provide criteria for choosing the approximate step $\hat{x}^k$ and controlling the initial parameter $\sigma$ in Algorithm \ref{globalaNM}. In practice, these requirements are not restrictive. Obviously, \eqref{newconditionforhatx} is satisfied when $\hat{x}^k = x^k$.  
In addition, if we choose $
\hat{x}^k = x^k - \lm \nabla \varphi(x^k)$, 
where $\lm>0$, then Algorithm \ref{globalaNM} reduces to the extragradient Newton method discussed in Remark \ref{choicestep}. In particular, the parameters $\lm>0$ and $\sigma>0$ can be chosen as follows:
$$
0< \lm < \left(\frac{\alpha L}{2\beta \xi} +L \right)^{-1} \quad \text{and }\;   0 < \sigma < 1 - \frac{\alpha L}{2\beta \xi (\lm^{-1}-L)}.
$$
In this case \eqref{newconditionforhatx} and \eqref{constant} are satisfied with $c = 1/\lm.$ 

\end{Remark}

\section{Local Superlinear and Quadratic Convergence Analysis}\label{sec:localconvergence}
In this section, we analyze the local fast convergence of our proposed method. In particular, we show that the convergence is superlinear or quadratic when the Hessian is positive definite at the accumulation point, which is a sufficient optimality condition for a stationary point and is weaker than the strong convexity assumption typically required in classical globalized Newton-type methods. To proceed, we recall the following proposition concerning the acceptance of the unit stepsize under semismoothness. This result follows from a careful examination of the proof of \cite[Theorem~3.3]{F1996} and was restated in \cite[Proposition 4.4]{kmptjogo} and  \cite[Proposition~2]{kmptmp}.  For convenience, we present it here in the special case where the function is twice continuously differentiable.

\begin{Proposition}[\bf acceptance of unit stepsize]\label{acceptancesm} Suppose that a function $\varphi:\R^n\to\R$ is $\mathcal{C}^{2}$-smooth around $\ox$ with $\nabla \varphi(\ox)=0$. Let a sequence $\{x^k\}$ converge to $\bar{x}$ with $x^k\ne\bar{x}$ as $k\in\N$, and let a sequence $\{d^k\}$ satisfy the condition 
$$
\|x^k+d^k-\bar{x}\|=o(\|x^k-\bar{x}\|).
$$ 
Suppose further that there exists $\kappa>0$ such that
$\langle\nabla\varphi(x^k),d^k\rangle\le\kappa^{-1}\|d^k\|^2$ whenever $k\in\N$ is sufficiently large. Then for any $\sigma\in (0,1/2)$ we have the estimate
$$
\varphi(x^k+d^k)\le\varphi(x^k)+\sigma\langle\nabla\varphi(x^k),d^k\rangle\;\mbox{ for all large }\;k\in \N.
$$
\end{Proposition}

Next, we establish the local superlinear and quadratic convergence rates of the iterates generated by Algorithm~\ref{globalaNM} under the assumption that the Hessian is positive definite at the accumulation point. 

\begin{Theorem}[\bf convergence under the positive definiteness of Hessian at accumulation points]\label{converHess} In the setting of Theorem~\ref{globalconver}, assume that $\inf \varphi(x^k) >-\infty$ and  \eqref{approximates} holds. Let $\bar{x}$ be an accumulation point of the sequence $\{x^k\}$ in which the Hessian $\nabla^2 \varphi(\bar{x})$ is positive definite. Then the entire sequence $\{x^k\}$ converges to $\bar{x}$, which is a strong local minimizer of $\varphi$. Furthermore, if there exist $\eta>0$ and $k_0 \in \N$ such that 
\begin{equation}\label{approximates2}
\norm{\hat{x}^k-\ox} \leq \eta \norm{x^k-\ox} \quad \text{for all }\; k \geq k_0,
\end{equation}
then the convergence rate of $\{x^k\}$ is Q-superlinear. Moreover, the convergence rate is quadratic if, in addition, $\nabla^2 \varphi$ is locally Lipschitz continuous in a neighborhood of $\bar{x}$. 
\end{Theorem}
\begin{proof}
Suppose  that $\{x^k\}$ has  {an accumulation point} $\bar{x}$ in which    $\nabla^2\varphi(\bar{x})$ is positive-definite. By  Theorem \ref{globalconver}, we deduce that $\ox$ is a stationary point.  Proposition \ref{uniformPD} gives us positive numbers $\kappa$ and $\delta$ such that
\begin{equation}\label{uniformPDHess}
\langle \nabla^2\varphi(x)w,w\rangle\ge\kappa\|w\|^2\quad\text{for all }\; x\in B_\delta(\bar{x}),\;\mbox{ and }\;w\in\R^n.
\end{equation}
Therefore, $\varphi$ is strongly convex on $B_\delta(\ox)$, which implies that $\ox$ is the unique stationary point of $\varphi$ in $B_\delta(\ox)$ and
$$
\varphi(x) \geq \varphi(\ox) + \langle \nabla\varphi(\ox),x-\ox\rangle + \frac{\kappa}{2}\|x-\ox\|^2 = \varphi(\ox)+ \frac{\kappa}{2}\|x-\ox\|^2  \quad \text{for all }\; x \in B_\delta(\ox).
$$
The latter means that $\ox$ is a strong local minimizer of $\varphi.$ 
To verify the convergence of $\{x^k\}$ to $\ox$, we only need to show that $\ox$ is isolated. Indeed,  we   show that no other  accumulation point of $\{x^k\}$ exists in ${B}_\delta(\bar{x})$. Assuming the contrary, find $\tilde x\in {B}_\delta(\bar{x})$ such that $\Tilde x\ne\bar{x}$ and that $\tilde x$ is  an accumulation point  of $\{x^k\}$. Using again Theorem \ref{globalconver}, $\tilde{x}$ is a stationary point of $\varphi$, which is a contradiction. Hence, by Theorem \ref{globalconver}, $\{x^k\}$ must converge to $\ox$ as $k \to \infty.$ 

\medskip We next verify the superlinear convergence of $\{x^k\}$. Indeed, since $x^k \to \ox$ as $k\to \infty$ and \eqref{approximates} holds, it follows that $\hat{x}^k \to \ox$ as $k \to \infty$.  Next, we need to show that
\begin{equation}\label{dksuperlinear}
\norm{\hat{x}^k + d^k -\ox} =o\left(\norm{\hat{x}^k -\ox}\right).
\end{equation}
Indeed, for each $k \in \N$, we have 
$$
-\nabla\varphi(\hat{x}^k)=\nabla^2\varphi(\hat{x}^k)d^k = \nabla^2\varphi(\hat{x}^k)(\hat{x}^k+d^k-\bar{x})+\nabla^2\varphi(\hat{x}^k)(-\hat{x}^k+\bar{x}),
$$
which implies that 
\begin{align}\label{estimatedksuper}
\left\|\hat{x}^k+d^k-\bar{x}\right\| & = \left\| -\nabla^2\varphi(\hat{x}^k)^{-1}\left(\nabla\varphi(\hat{x}^k) - \nabla^2\varphi(\hat{x}^k)(\hat{x}^k-\bar{x})\right)\right\| \nonumber\\
    & \leq \left\|\nabla^2\varphi(\hat{x}^k)^{-1} \right\| \left\| \nabla\varphi(\hat{x}^k) - \nabla\varphi(\ox)- \nabla^2\varphi(\hat{x}^k)(\hat{x}^k- \bar{x})\right\| \nonumber\\
    & \leq \frac{1}{\kappa} \left\| \nabla\varphi(\hat{x}^k) - \nabla\varphi(\ox)- \nabla^2\varphi(\hat{x}^k)(\hat{x}^k- \bar{x})\right\|
\end{align}
for all sufficiently large $k \in \N$. The last inequality follows from the fact that $\hat{x}^k \to \ox$ and the condition \eqref{uniformPDHess} holds. Since $\varphi$ is $\mathcal{C}^2$, we must have 
\begin{equation}\label{C2}
\left\| \nabla\varphi(\hat{x}^k) - \nabla\varphi(\ox)- \nabla^2\varphi(\hat{x}^k)(\hat{x}^k- \bar{x})\right\| = o\left(\norm{\hat{x}^k -\ox}\right).
\end{equation}
Combining \eqref{estimatedksuper} and \eqref{C2}, we obtain \eqref{dksuperlinear}. Moreover, since $\hat{x}^k \to \ox$ and \eqref{uniformPDHess} holds, we deduce that $\dotproduct{ \nabla^2\varphi(\hat{x}^k)d^k,d^k} \geq \kappa \norm{d^k}^2$ for all sufficiently large $k\in \N$. Due to the construction of $\{\hat{x}^k\}$ in Algorithm \ref{globalaNM}, we have $\dotproduct{\nabla \varphi(\hat{x}^k),d^k} \leq -\kappa \norm{d^k}^2$ for all sufficiently large $k\in \N$. Applying Proposition \ref{acceptancesm}, we obtain the following 
$$
\varphi(\hat{x}^k+d^k)\le\varphi(\hat{x}^k)+\sigma\dotproduct{\nabla\varphi(\hat{x}^k),d^k}\;\mbox{ for all large }\;k,
$$
which means that   $\tau_k=1$ for all $k$ sufficiently large, and thus we deduce from \eqref{dksuperlinear} that
$$
\norm{x^{k+1} - \ox} = \norm{\hat{x}^k + d^k -\ox}  =o\left(\norm{\hat{x}^k-\ox}\right).
$$
Combining the latter with  \eqref{approximates2}, we obtain the superlinear convergence of $\{x^k\}.$  Suppose further that $\nabla^2\varphi$ is locally Lipschitz in a neighborhood of $\ox$. Then there exists $M>0$ such that  
\begin{equation}\label{C2,1}
\left\| \nabla\varphi(\hat{x}^k) - \nabla\varphi(\ox)- \nabla^2\varphi(\hat{x}^k)(\hat{x}^k- \bar{x})\right\| \leq  M\norm{\hat{x}^k -\ox}^2 \leq M\eta \norm{x^k -\ox}^2 
\end{equation}
for sufficiently large $k \in \N.$ Combining \eqref{estimatedksuper} and \eqref{C2,1}, we obtain the quadratic convergence of $\{x^k\},$ which completes the  proof.
\end{proof}

\begin{Remark}[\bf on the condition in Theorem \ref{converHess}] \rm Condition~\eqref{approximates2} is also satisfied by both the hybrid gradient--Newton method and the proposed extragradient Newton method. Indeed, when $\hat{x}^k=x^k$, Algorithm~\ref{globalaNM} reduces to the hybrid gradient--Newton method, and condition~\eqref{approximates2} holds trivially. On the other hand, when $
\hat{x}^k=x^k-\lambda\nabla\varphi(x^k)$, where $\lambda>0$, 
Algorithm~\ref{globalaNM} becomes the proposed extragradient Newton method. In this case, condition~\eqref{approximates2} is satisfied for all sufficiently large $k$ due to the local Lipschitz continuity of $\nabla\varphi$ around $\bar{x}$. 
\end{Remark}

\section{Convergence Analysis under  Quasiconvexity}\label{sec:quasiconvex}
This section provides an additional convergence analysis of our proposed method under generalized convexity, thereby significantly broadening the class of problems to which our Newton-type method can be applied.  Because convexity does not cover many models arising in science, economics, and engineering, several extensions have been proposed in the literature \cite{Crouzeix98,d49,m65}, including quasiconvex and pseudoconvex functions, which retain many of the desirable properties of convex functions. In particular, quasiconvex functions ensure that all sublevel sets are convex, whereas pseudoconvex functions guarantee that every critical point is a global minimizer. We now recall these concepts in more detail.   Let $\varphi: \R^n \to  {\R}$ be a function, and a nonempty convex set $\Omega \subset \R^n$. The function $\varphi$ is called \textit{quasiconvex} on $\Omega$ if
$$
\varphi((1-\lambda)x + \lambda y) \leq \max\{\varphi(x), \varphi(y)\},
$$
for all $x, y \in \Omega$ and all $\lambda \in [0,1]$. The function $\varphi$ is called \textit{pseudoconvex} on $\Omega$ if it is differentiable and satisfies
$$
x,y \in \Omega, \quad \varphi(x) > \varphi(y) \Longrightarrow \langle \nabla \varphi(x), y - x \rangle < 0.
$$
From these definitions, it is clear that every differentiable pseudoconvex function is also quasiconvex. More properties and characterizations of these classes of functions can be found in \cite{kp18,kp20}.    
Next, we revisit the class of \textit{strongly quasiconvex functions}, introduced by Boris Polyak in his seminal 1966 paper \cite{polyak66}. More precisely, a function $\varphi$ is called \textit{strongly quasiconvex} with modulus $\alpha > 0$ on $\Omega$ if
\begin{equation}
\varphi(\lambda x + (1-\lambda)y) \leq \max\{\varphi(x), \varphi(y)\} - \alpha \lambda(1-\lambda)\|x - y\|^2, 
\end{equation}
for all $x, y \in \Omega$ and all $\lambda \in [0,1]$.  It follows from \cite[Proposition~15]{l22} that every strongly quasiconvex function is pseudoconvex, and therefore also quasiconvex. Consequently, strongly quasiconvex functions inherit all the properties possessed by pseudoconvex and quasiconvex functions. 
This notion has been extensively studied from theoretical and practical perspectives, including first-order characterizations, dynamical systems exhibiting this property, and applications in numerical methods such as gradient descent, heavy ball, and inertial algorithms, as highlighted in recent works \cite{glm25,l22,l25}. In this section, we focus on how these generalized notions of convexity influence the convergence and convergence rates of the Newton-type methods proposed in this paper. To proceed, we first present a second-order necessary condition for  strongly quasiconvex functions. 

\begin{Proposition}[\bf second-order necessary condition for strong quasiconvexity]\label{2ndness} Let $\varphi:\R^n\to \R$ be $\mathcal{C}^2$-smooth on an open convex set $\Omega \subset \R^n.$ If $\varphi$ is strongly quasiconvex with modulus $\alpha>0$, then we have 
\begin{equation}\label{2ndquasi}
x \in \Omega, w \in \R^n, \langle \nabla\varphi(x),w\rangle =0 \Longrightarrow \langle \nabla^2 \varphi(x)w,w\rangle \geq 2\alpha\|w\|^2.
\end{equation}
\end{Proposition}
\begin{proof} Let $x \in \Omega$, $w \in \R^n\setminus \{0\}$ be such that $\langle \nabla\varphi(x),w\rangle =0$. Choosing any sequence $\{\tau_k\} \subset (0,\infty)$ such that $\tau_k \to 0$, we consider the sequences $\{x_k\}, \{y_k\}$ as follows:
$$
x_k:= x + \tau_k w, \quad y_k:= x -\tau_kw, \quad k \in \N. 
$$
Due to the strong quasiconvexity of $\varphi$ with modulus $\alpha$ on $\Omega$, we have 
$$
\varphi(x) = \varphi\left(\frac{1}{2} x_k + \frac{1}{2}y_k \right) \leq \max\{\varphi(x_k),\varphi(y_k)\} - \frac{\alpha}{4}\|x_k-y_k\|^2 \quad \text{for all sufficiently large } k,
$$
which implies that 
\begin{equation}\label{max>}
 \max\{\varphi(x_k)-\varphi(x),\varphi(y_k)-\varphi(x)\}  \geq  \alpha \tau_k^2 \|w\|^2\quad \text{for all sufficiently large } k. 
\end{equation}
By Taylor expansion, we have 
$$
\lim_{k\to \infty} \frac{\varphi(x_k)-\varphi(x)-\tau_k\langle\nabla \varphi(x),w\rangle -\frac{1}{2}\tau_k^2\langle \nabla^2\varphi(x)w,w\rangle}{\tau_k^2\|w\|^2}=0. 
$$
$$
\lim_{k\to \infty} \frac{\varphi(y_k)-\varphi(x)+\tau_k\langle\nabla \varphi(x),w\rangle -\frac{1}{2}\tau_k^2\langle \nabla^2\varphi(x)w,w\rangle}{\tau_k^2\|w\|^2}=0. 
$$
Since $\langle \nabla\varphi(x),w\rangle =0$, we deduce from the latter that 
\begin{equation}\label{maxxkyk}
\lim_{k\to\infty} \frac{\varphi(x_k)-\varphi(x)}{\tau_k^2}=\lim_{k\to\infty} \frac{\varphi(y_k)-\varphi(x)}{\tau_k^2} = \frac{1}{2}\langle \nabla^2\varphi(x)w,w\rangle. 
\end{equation}
Combining \eqref{max>} and \eqref{maxxkyk}, we obtain \eqref{2ndquasi}. The proof is complete. 
    
\end{proof}

We now study the convergence of Algorithm~\ref{globalaNM} under the assumption that the objective function is strongly quasiconvex. In this setting, the existence of a solution is guaranteed, and the generated iterative sequence converges globally to the unique solution, which satisfies the global quadratic growth condition.

\begin{Theorem}[\bf convergence under strong quasiconvexity]  \label{converquasi}
Let $\varphi \colon \R^n \to \R$ be a $\mathcal{C}^2$-smooth function that is strongly quasiconvex with modulus $\alpha>0$. Then Algorithm~{\rm\ref{globalaNM}} either terminates after finitely many iterations or generates a sequence $\{x^k\}$ such that the sequence  of function values $\{\varphi(x^k)\}$ is monotonically decreasing. Then we have the following assertions:
\begin{itemize}
    \item[\bf (i)] If the sequence $\{\hat{x}^k\}$ in Algorithm~{\rm\ref{globalaNM}} satisfies  \eqref{approximates}, 
then the sequence $\{x^k\}$ converges to $\ox$, which is the unique global minimizer of $\varphi$ satisfying the global growth condition
\begin{equation}\label{growth}
\varphi(x) \geq \varphi(\ox) + \frac{\alpha}{4}\|x-\ox\|^2 \quad \text{for all }\; x \in \R^n.
\end{equation}
\item[\bf (ii)]  If there exist constants $\eta>0$ and $k_0 \in \N$ such that
$$
\norm{\hat{x}^k - \ox} \leq \eta \norm{x^k - \ox} \quad \text{for all } k \geq k_0,
$$
then the sequence $\{x^k\}$ converges superlinearly to $\ox$ as $k \to \infty$. Furthermore, if $\nabla^2 \varphi$ is locally Lipschitz continuous in a neighborhood of $\ox$, the convergence of $\{x^k\}$ is quadratic.
\end{itemize}

\end{Theorem}
\begin{proof} Let $x^0 \in \R^n$ be the starting point of Algorithm~\ref{globalaNM}, and define the level set
$$
\Omega := \{x \in \R^n \mid \varphi(x) \leq \varphi(x^0)\}.
$$
By Theorem~\ref{globalconver}, Algorithm~{\rm\ref{globalaNM}} either terminates after finitely many iterations or generates a sequence $\{x^k\}$ such that the function values $\{\varphi(x^k)\}$ form a monotonically decreasing sequence. This implies that $
\{x^k\} \subset \Omega$ for all   $k \in \N.$ 
Moreover, it follows from \cite[Theorem 1]{l22} that $\varphi$ is coercive, so $\Omega$ is bounded. Consequently, the sequence $\{x^k\}$ is bounded and therefore possesses at least one accumulation point.   
Applying Theorem~\ref{globalconver} again, any accumulation point $\ox$ is a stationary point of $\varphi$. Then, by the strong quasiconvexity of $\varphi$ and Proposition~\ref{2ndness}, the Hessian $\nabla^2 \varphi(\ox)$ is positive definite, and $\ox$ is an isolated accumulation point.  Thus, the sequence $\{x^k\}$ must converge to $\ox$. Moreover, the strong quasiconvexity of $\varphi$ implies that $\varphi$ is pseudoconvex, so the stationary point $\ox$ is a global minimizer by \cite[Theorem 3.2.5]{AlbertoLaura09}. The uniqueness of $\ox$ follows from \cite[Corollary~3]{l22}, and the global growth condition \eqref{growth} is established by \cite[Corollary~9]{hl25}, which clarifies the assertion {\bf (i)}. We next verify the assertions {\bf (ii)}. Indeed,  Theorem~\ref{converHess} guarantees that the sequence $\{x^k\}$ converges superlinearly to $\ox$, with the convergence rate being quadratic if $\nabla^2 \varphi$ is locally Lipschitz in a neighborhood of $\ox$.  The proof is complete.

\end{proof}

\section{Numerical Experiments}\label{sec:num}
    In this section we test the extragradient Newton method method against the hybrid gradient-Newton method, along with the classical first order methods of gradient descent, heavyball and Nesterov's accelerated method. We do not tune the methods to each individual example, and instead use the parameter free versions of each individual method utilizing backtracking line searches. That is, in the updates for gradient descent 
    \begin{equation} \label{eqn: gradient descent step}
        x^{k+1}  = x^k + \tau_k d^k, \quad d^k = -\nabla \varphi(x^k), 
    \end{equation}
    heavyball 
    \begin{equation} \label{eqn: heavy ball step}
        x^{k+1} = x^k + \alpha_k (x^k - x^{k-1}) + \tau_k d^k, \quad  d^k = - \nabla \varphi (x^k)
    \end{equation}
    and Nesterov's accelerated method 
    \begin{equation} \label{eqn: nesterov step}
        x^{k+1} = x^k + \alpha_k (x^k - x^{k-1}) + \tau_kd^k, \quad d^k =-\nabla \varphi (x^k + \alpha_k (x^k - x^{k-1}) ) 
    \end{equation}
we compute $\tau_k:= \text{\rm argmax} \left\{\tau|\; \varphi(\hat{x}^k +\tau d^k) \leq  \varphi(x^k) + \sigma \tau \dotproduct{\nabla\varphi(x^k),d^k}   , \tau =\beta^{i}, i \in \N \right\}$ for parameters $\sigma = 1/2$ and $\beta = 1/2$. These values are sufficient to obtain convergence in the convex case \cite{Nesterov-W}. For Nesterov's accelerated method we choose $\alpha_k = \frac{k-1}{k+2}$ as standard \cite{Nesterov-W}, and for heavyball we choose a constant $\alpha_k = 0.3$. The linesearch in the Newton steps does not require such a large value for $\sigma$ and thus we use $\sigma = 0.05$ and $\beta = 1/2$. In the extragradient Newton method we choose $\hat{x}^k$ as a step of gradient descent with the same parameters as above. In Newton-type algorithms, we use $\zeta = 10^{-6}$. All examples are solved until $\|\nabla \varphi(x^k) \| \leq 10^{-4}$. 

These examples are chosen to demonstrate realistic problem regimes in which the extragradient Newton method offers improvement. There is no claim that the extragradient method is universally superior to standard hybrid gradient-Newton, however this does demonstrate that the framework of Algorithm \ref{globalaNM} allows one to adapt hybrid gradient-Newton to particular problem types for increased performance.

All examples are run on Matlab 2025a on a 3.8 GHz 8-Core Intel Core i7 with 32 GB 2667 MHz DDR4 memory. A seed of 47 is used throughout the examples.

\begin{Example} \rm \label{example: quasiconvex 1D} We consider the simple one dimensional quasiconvex objective function \cite{glm25}
\begin{equation*}
    \varphi(x) = x^2 + 3\sin(x)^2.
\end{equation*}
It has a single global minimum at 0. A comparison of the iterations for gradient descent, hybrid gradient-Newton and extragradient Newton method with an initial value of 6 is given in Figure \ref{fig: quasiconvex 1D objective plot}. We can see the benefit of the extragradient Newton method, as a single iteration of is equivalent to two iterations of hybrid gradient-Newton, and does not overshoot the minimum.  

Table \ref{table: quasiconvex 1D} shows the performance of the various methods. The runtime was averaged over 1000 runs, to account for timing fluctuations due to the fast run times. The extragradient Newton method outperforms all both iterations and runtime, halving the number of iterations of the hybrid gradient-Newton.
\begin{figure}[H]
    \centering
    \includegraphics[width=0.5\linewidth]{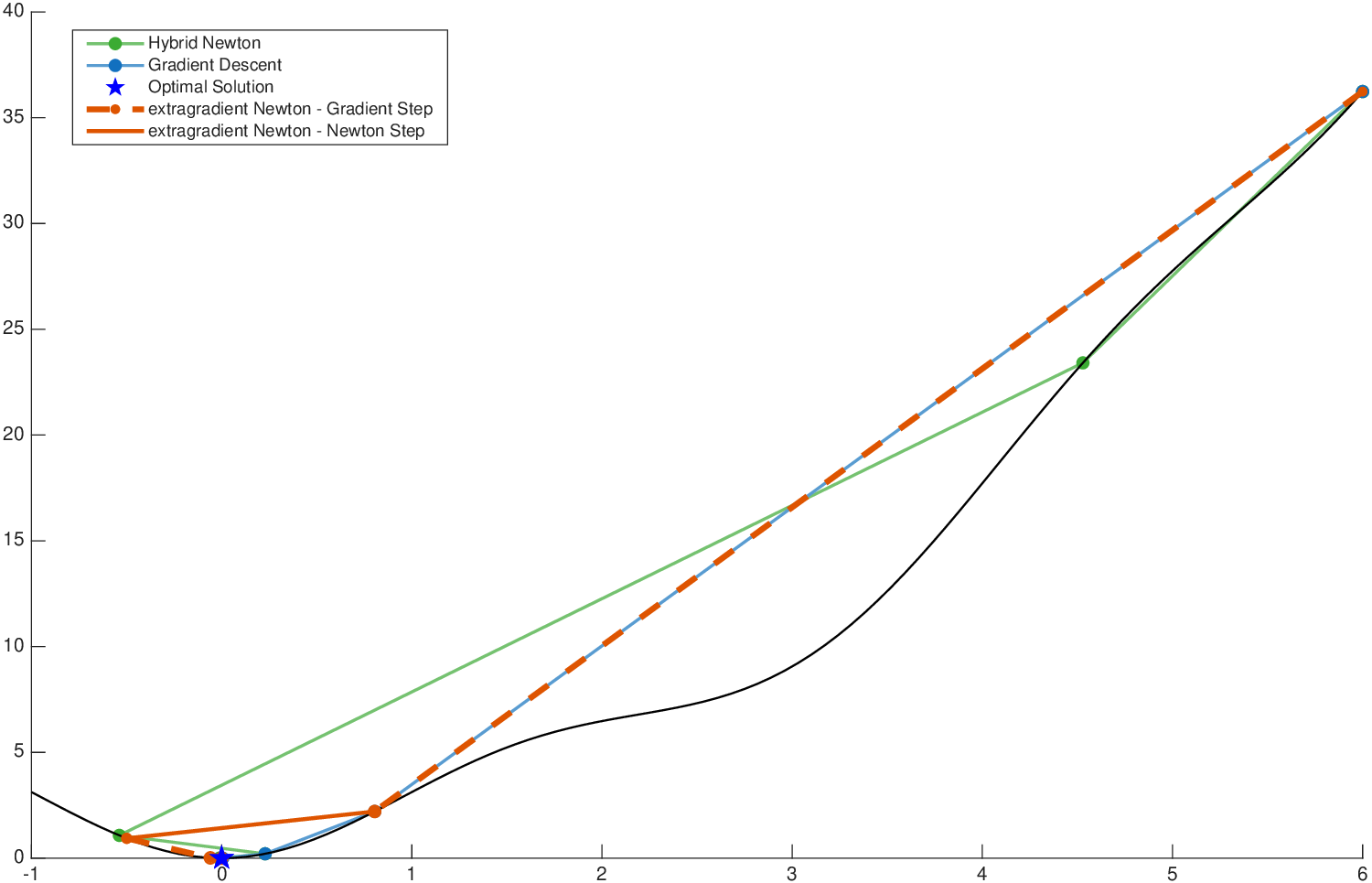}
    \caption{Illustration of the iterations the algorithms applied to Example \ref{example: quasiconvex 1D}.}
    \label{fig: quasiconvex 1D objective plot}
\end{figure}

\begin{table}[H]
\centering
\begingroup \scriptsize %
\begin {tabular}{c|ccccc}%
\toprule &gradient method&heavyball method&\shortstack {Nesterov\\ accelerated method}&\shortstack {extragradient \\ Newton method}&\shortstack {hybrid\\gradient-Newton}\\\midrule %
iterations&\pgfutilensuremath {5}&\pgfutilensuremath {17}&\pgfutilensuremath {7}&\pgfutilensuremath {3}&\pgfutilensuremath {8}\\%
time (s)&\pgfutilensuremath {1.34\cdot 10^{-3}}&\pgfutilensuremath {3.00\cdot 10^{-3}}&\pgfutilensuremath {1.33\cdot 10^{-3}}&\pgfutilensuremath {9.46\cdot 10^{-4}}&\pgfutilensuremath {1.55\cdot 10^{-3}}\\%
\end {tabular}%
\endgroup %

\caption{Performance comparison for Example \ref{example: quasiconvex 1D}. Time was averaged over 1000 runs. }
\label{table: quasiconvex 1D}
\end{table}

\end{Example}

\begin{Example} \rm \label{example: quasiconvex}

Consider the quasiconvex programming
$$
\min \quad \varphi(x) = \log \left( \|x\|^2_A +1 \right) + \frac{\mu}{2}\|x\|_A^2  \quad \text{subject to }\;  x \in \R^N,
$$
where the  {\em energy norm} $\|\cdot\|_A$ is given by
$$
\|x\|_A:= \langle Ax,x\rangle \quad \text{for all }\; x \in \R^N, 
$$
with a positive definite  matrix $A \in \R^{N \times N}$ randomly generated with eigenvalues linearly spaced in $[1,100]$ to create an ill-conditioned problem. 
The cost function $\varphi$ is nonconvex, but it is strongly quasiconvex, and thus the global convergence  of Algorithm \ref{globalaNM} is guaranteed due to Theorem \ref{converquasi}. 
\begin{proof} We can decompose $\varphi(x)$ into two parts, \(\varphi(x) = \varphi_1(x)+\varphi_2(x)\), with $\varphi_1(x) = \log \left( \|x\|^2_A +1 \right)$ and $\varphi_2(x) = \frac{\mu}{2}\|x\|_A^2$. Now $\varphi_2(x)$ is strongly convex with modulus $\gamma = \mu \lambda_{\text{min}}(A)$ and thus strongly quasiconvex with the same modulus. Moreover, $\varphi_1$ is quasiconvex because  it is the composition of a nondecreasing and a quasiconvex function \cite{GP1971}. Take $x,y \in \R^N$ and $z\coloneqq(1-\lambda)x + \lambda y$ for $\lambda \in [0,1]$.  Without loss of generality suppose $\varphi_2(x) \leq \varphi_2(y)$. This implies that
$$
\log \left( \|x\|^2_A +1 \right) \leq \log \left( \|y\|^2_A +1 \right),
$$
and thus $\varphi_1(x) \leq \varphi_1(y)$. Combining the latter with the quasiconvexity of $\varphi_1$, we deduce that $\varphi_1(z) \leq \varphi_1(y)$. Therefore, 
\begin{align*}
    \varphi(z) &= \varphi_1(z) + \varphi_2(z) \\
    &\leq \varphi_1(y) + \max \{ \varphi_2(x), \varphi_2(y) \} - \frac{\gamma}{2} \lambda(1-\lambda) \|y-x\|^2 \\
    &\leq \varphi(y) - \frac{\gamma}{2} \lambda(1-\lambda) \|y-x\|^2,
\end{align*}
which implies that $\varphi$ is strongly quasiconvex. 
\end{proof}

There is a global minimizer at $x=0$. For the experiment, we choose $\mu = 10^{-4}$ and initial values are randomly chosen with entries drawn uniformly distributed from $[0,10]$. We solve the problem for various  of $N$ with results collected in Table \ref{table: quasiconvex}. Convergence behavior for $N=1024$ can be seen in Figure \ref{fig: quasiconvex 1024 convergence}.

The Newton methods vastly outperform the first-order methods due to the poor conditioning of the problem. The extragradient Newton method method is dominant in runtime and iteration count throughout all scales. It has superior iterations over the hybrid gradient-Newton method due primarily to the extra gradient step assisting in more rapidly moving the near the minimizer from the relatively far initial value. 

\begin{table}[h]
\centering
\begingroup \scriptsize %
\begin {tabular}{c|cccccccccc}%
\toprule & \multicolumn {2}{c}{gradient method} & \multicolumn {2}{c}{heavyball method} &\multicolumn {2}{c}{\shortstack {Nesterov \\ accelerated method}} & \multicolumn {2}{c}{ \shortstack {extragradient \\ Newton method}} & \multicolumn {2}{c}{\shortstack {hybrid \\ gradient-Newton}}\\N&iterations&time&iterations&time&iterations&time&iterations&time&iterations&time\\\midrule %
\pgfutilensuremath {2}&\pgfutilensuremath {355}&\pgfutilensuremath {0.063}&\pgfutilensuremath {110}&\pgfutilensuremath {0.018}&\pgfutilensuremath {225}&\pgfutilensuremath {0.038}&\pgfutilensuremath {24}&\pgfutilensuremath {0.008}&\pgfutilensuremath {75}&\pgfutilensuremath {0.016}\\%
\pgfutilensuremath {4}&\pgfutilensuremath {376}&\pgfutilensuremath {0.064}&\pgfutilensuremath {165}&\pgfutilensuremath {0.027}&\pgfutilensuremath {334}&\pgfutilensuremath {0.055}&\pgfutilensuremath {24}&\pgfutilensuremath {0.005}&\pgfutilensuremath {45}&\pgfutilensuremath {0.008}\\%
\pgfutilensuremath {8}&\pgfutilensuremath {427}&\pgfutilensuremath {0.071}&\pgfutilensuremath {150}&\pgfutilensuremath {0.025}&\pgfutilensuremath {433}&\pgfutilensuremath {0.081}&\pgfutilensuremath {42}&\pgfutilensuremath {0.010}&\pgfutilensuremath {101}&\pgfutilensuremath {0.018}\\%
\pgfutilensuremath {16}&\pgfutilensuremath {450}&\pgfutilensuremath {0.074}&\pgfutilensuremath {226}&\pgfutilensuremath {0.037}&\pgfutilensuremath {1{,}013}&\pgfutilensuremath {0.168}&\pgfutilensuremath {43}&\pgfutilensuremath {0.010}&\pgfutilensuremath {88}&\pgfutilensuremath {0.016}\\%
\pgfutilensuremath {32}&\pgfutilensuremath {462}&\pgfutilensuremath {0.077}&\pgfutilensuremath {239}&\pgfutilensuremath {0.040}&\pgfutilensuremath {1{,}100}&\pgfutilensuremath {0.185}&\pgfutilensuremath {27}&\pgfutilensuremath {0.006}&\pgfutilensuremath {60}&\pgfutilensuremath {0.011}\\%
\pgfutilensuremath {64}&\pgfutilensuremath {563}&\pgfutilensuremath {0.100}&\pgfutilensuremath {303}&\pgfutilensuremath {0.053}&\pgfutilensuremath {812}&\pgfutilensuremath {0.140}&\pgfutilensuremath {39}&\pgfutilensuremath {0.011}&\pgfutilensuremath {75}&\pgfutilensuremath {0.016}\\%
\pgfutilensuremath {128}&\pgfutilensuremath {689}&\pgfutilensuremath {0.164}&\pgfutilensuremath {402}&\pgfutilensuremath {0.094}&\pgfutilensuremath {1{,}357}&\pgfutilensuremath {0.304}&\pgfutilensuremath {29}&\pgfutilensuremath {0.015}&\pgfutilensuremath {58}&\pgfutilensuremath {0.020}\\%
\pgfutilensuremath {256}&\pgfutilensuremath {925}&\pgfutilensuremath {0.366}&\pgfutilensuremath {531}&\pgfutilensuremath {0.198}&\pgfutilensuremath {2{,}021}&\pgfutilensuremath {0.762}&\pgfutilensuremath {27}&\pgfutilensuremath {0.029}&\pgfutilensuremath {60}&\pgfutilensuremath {0.040}\\%
\pgfutilensuremath {512}&\pgfutilensuremath {1{,}003}&\pgfutilensuremath {0.605}&\pgfutilensuremath {580}&\pgfutilensuremath {0.321}&\pgfutilensuremath {1{,}562}&\pgfutilensuremath {0.888}&\pgfutilensuremath {15}&\pgfutilensuremath {0.036}&\pgfutilensuremath {32}&\pgfutilensuremath {0.058}\\%
\pgfutilensuremath {1{,}024}&\pgfutilensuremath {1{,}187}&\pgfutilensuremath {1.868}&\pgfutilensuremath {758}&\pgfutilensuremath {1.171}&\pgfutilensuremath {2{,}166}&\pgfutilensuremath {3.029}&\pgfutilensuremath {11}&\pgfutilensuremath {0.095}&\pgfutilensuremath {24}&\pgfutilensuremath {0.169}\\%
\pgfutilensuremath {2{,}048}&\pgfutilensuremath {1{,}648}&\pgfutilensuremath {28.816}&\pgfutilensuremath {1{,}082}&\pgfutilensuremath {18.207}&\pgfutilensuremath {2{,}090}&\pgfutilensuremath {34.650}&\pgfutilensuremath {10}&\pgfutilensuremath {0.761}&\pgfutilensuremath {18}&\pgfutilensuremath {0.980}\\%
\pgfutilensuremath {4{,}096}&\pgfutilensuremath {2{,}028}&\pgfutilensuremath {160.026}&\pgfutilensuremath {1{,}319}&\pgfutilensuremath {100.934}&\pgfutilensuremath {2{,}066}&\pgfutilensuremath {166.713}&\pgfutilensuremath {8}&\pgfutilensuremath {3.009}&\pgfutilensuremath {16}&\pgfutilensuremath {4.866}\\%
\end {tabular}%
\endgroup %

\caption{Performance comparison for Example \ref{example: quasiconvex}. Time is measured in seconds.}
\label{table: quasiconvex}
\end{table}

\begin{figure}
    \centering
        \includegraphics[width=0.85\linewidth]{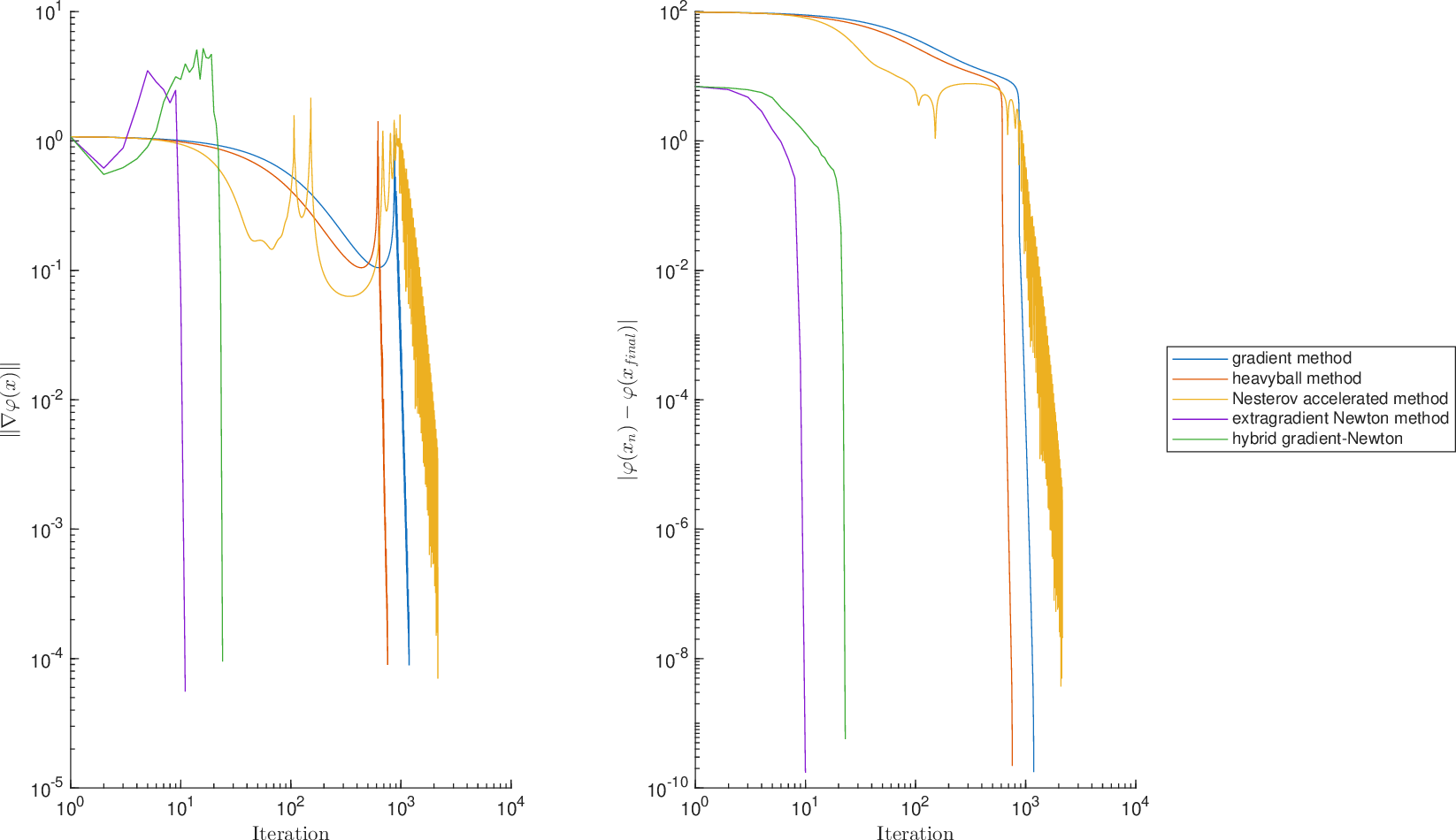}
    \caption{Convergence curves for Example \ref{example: quasiconvex} for $N = 1024$.}
    \label{fig: quasiconvex 1024 convergence}
\end{figure}
\end{Example}

\begin{Example}\rm  \label{example: Broyden Banded}

We now examine the Broyden banded function \cite{Broyden,unconstraints} as implemented in the \cite{s2mpj} package. This function is given for $x \in \R^N$ as 
\begin{equation}\label{eqn: broyden banded function}
\begin{split} 
    f_i(x) &= x_i(2+5x_i^2) + 1 - \sum_{j \in J_i} x_j(1+x_j) \\
    &\text{where } J_i = \{j \; | \; j \neq i, \; \max(1, i-m_l) \leq j \leq \text{min}(n,1+m_u)  \}.
\end{split}
\end{equation}
We specifically consider the parameters $m_l = 5$, $m_u = 1$ and $N=10$.  We solve for a root of this function by solving
\[
\min \quad \varphi(x) = \|f(x)\|_2^2. 
\]
The resulting problem is nonconvex but smooth, with a banded structure. There are multiple roots and we choose an initial value of $(1,\hdots, 1)$ which converges to the root at 
(\numlist[
round-mode=places, 
round-precision=3,
list-final-separator = {, },
]{0.267509316275924;0.438474888764556;0.599720708888636;0.748250087236817;0.882734032075959;0.989082189392820;1.160414943044660;1.308592592296775;1.287850864246086;1.243604886924583}). This is a non-standard initial value, the standard one being $(-1, \hdots, -1)$ which is a "nice" initial value \cite{Broyden}. We wish to test the benefit of the extragradient Newton method which occurs when "poor" initial values are chosen, hence the nonstandard initial value. Table \ref{table: broyden banded} shows that this choice causes poor performance for hybrid gradient-Newton, while extragradient Newton method outperforms it by a factor of 100. The convergence behavior of the algorithms can be seen in Figure \ref{fig: broyden banded}.
\begin{table}[h]
\centering
\begingroup \scriptsize %
\begin {tabular}{c|ccccc}%
\toprule &gradient method&heavyball method&\shortstack {Nesterov\\ accelerated method}&\shortstack {extragradient \\ Newton method}&\shortstack {hybrid\\gradient-Newton}\\\midrule %
iterations&\pgfutilensuremath {1{,}886}&\pgfutilensuremath {2{,}140}&\pgfutilensuremath {476}&\pgfutilensuremath {11}&\pgfutilensuremath {1{,}019}\\%
time (s)&\pgfutilensuremath {24.888}&\pgfutilensuremath {27.748}&\pgfutilensuremath {3.497}&\pgfutilensuremath {0.248}&\pgfutilensuremath {15.626}\\%
\end {tabular}%
\endgroup %

\caption{Performance comparison for Example \ref{example: Broyden Banded}.}
\label{table: broyden banded}
\end{table}

\begin{figure}
    \centering
    \begin{subfigure}{0.4\linewidth}
        \includegraphics[width=\linewidth]{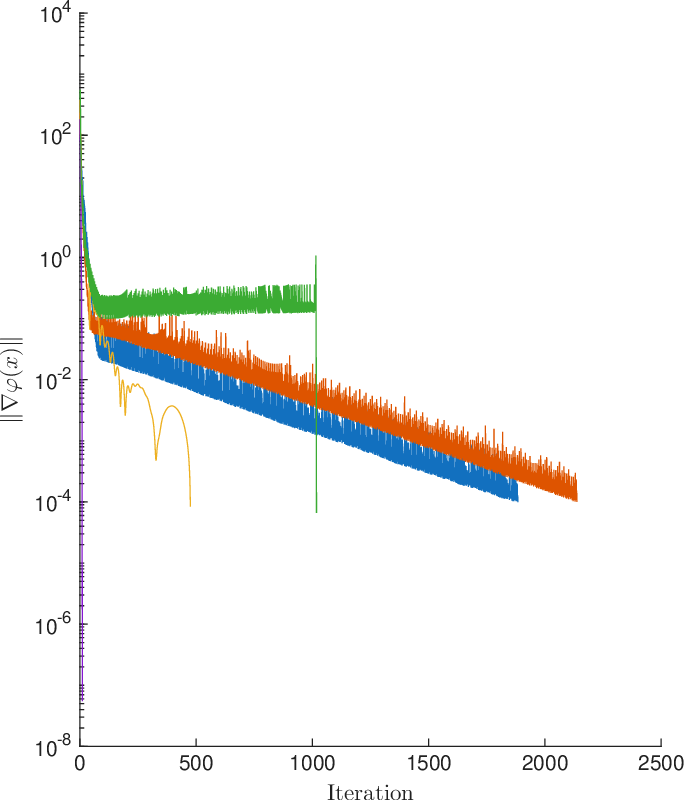}
    \end{subfigure}
    \begin{subfigure}{0.4\linewidth}
        \includegraphics[width=\linewidth]{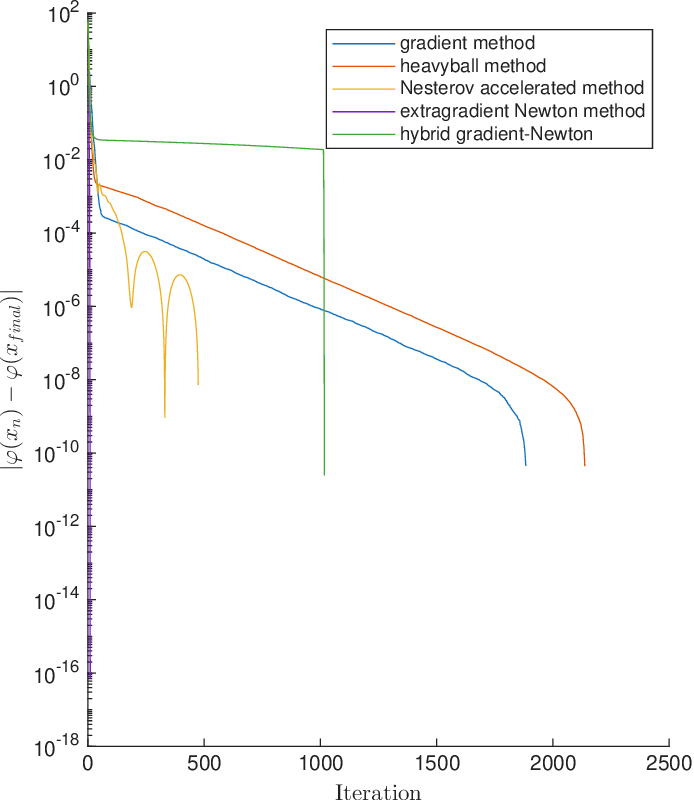}
    \end{subfigure}
    \caption{Convergence curves for Example \ref{example: Broyden Banded}.}
    \label{fig: broyden banded}
\end{figure}

\end{Example}

\begin{Example}\rm \label{example: box2}
    We now evaluate the algorithms on the Box two dimensional function \cite{box,unconstraints} as implemented in the \cite{s2mpj} package. This function is given for $x \in \mathbb{R}^2$ as 
\begin{align} \label{eqn: box3}
    f_i(x) &= e^{-t_i x_1} - e^{-t_ix_2} - \left(e^{-t_i} - e^{-10t_i} \right), \quad t_i = 0.1i, \quad i =1,2. 
\end{align}
As in the previous example we consider the minimization problem 
\[
\min \quad \varphi(x) = \|f(x)\|_2^2. 
\]
The resulting problem is nonconvex, smooth and highly nonlinear and asymmetric \cite{box}. It was originally motivated by estimating parameters from a pair of simultaneous linear differential equations. A plot of the objective contours is given in Figure \ref{fig: box2 contour}, which indicates a highly asymmetric curved valley increasing the problem difficulty. We use the initial value $(0,10)$. In Table \ref{table: box 2}, a summary of the performance is given, and in Figure \ref{fig: box 2} the convergence curves are plotted. We can see that the asymmetric curved valley greatly mislead the hybrid gradient-Newton method, causing it to take far more iterations than any other method, due to the local curvature directing the iterates away from the optimal solution. The extragradient method, does not suffer from this issue and converged quickly to the optimal solution, suggesting that this could be useful in similar problems with asymmetric curved valleys. 

\begin{figure}
    \centering
    \includegraphics[width=0.7\linewidth]{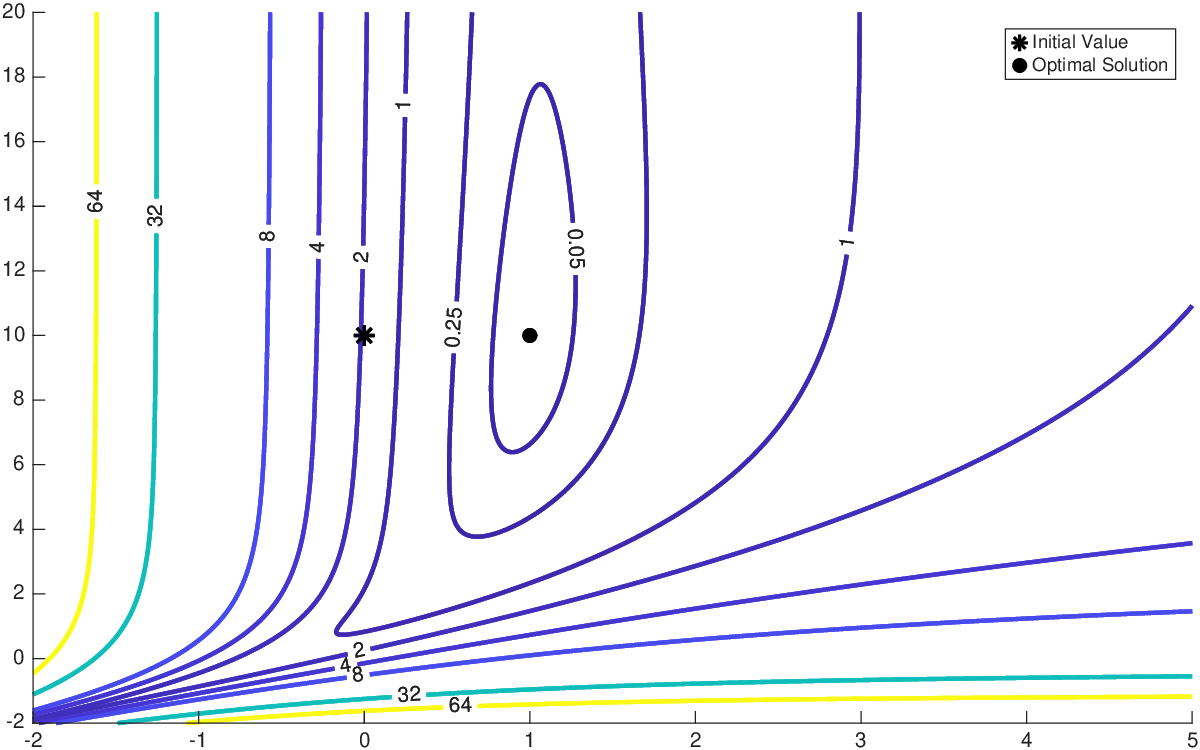}
    \caption{Contours of the exponential objective in Example \ref{example: box2}.}
    \label{fig: box2 contour}
\end{figure}

\begin{table}[h]
\centering
\begingroup \scriptsize %
\begin {tabular}{c|ccccc}%
\toprule &gradient method&heavyball method&\shortstack {Nesterov\\ accelerated method}&\shortstack {extragradient \\ Newton method}&\shortstack {hybrid\\gradient-Newton}\\\midrule %
iterations&\pgfutilensuremath {82}&\pgfutilensuremath {32}&\pgfutilensuremath {72}&\pgfutilensuremath {6}&\pgfutilensuremath {3{,}848}\\%
time (s)&\pgfutilensuremath {0.347}&\pgfutilensuremath {0.145}&\pgfutilensuremath {0.188}&\pgfutilensuremath {0.070}&\pgfutilensuremath {17.402}\\%
\end {tabular}%
\endgroup %

\caption{Performance comparison for Example \ref{example: box2}.}
\label{table: box 2}
\end{table}

\begin{figure}
    \centering
    \begin{subfigure}{0.4\linewidth}
        \includegraphics[width=\linewidth]{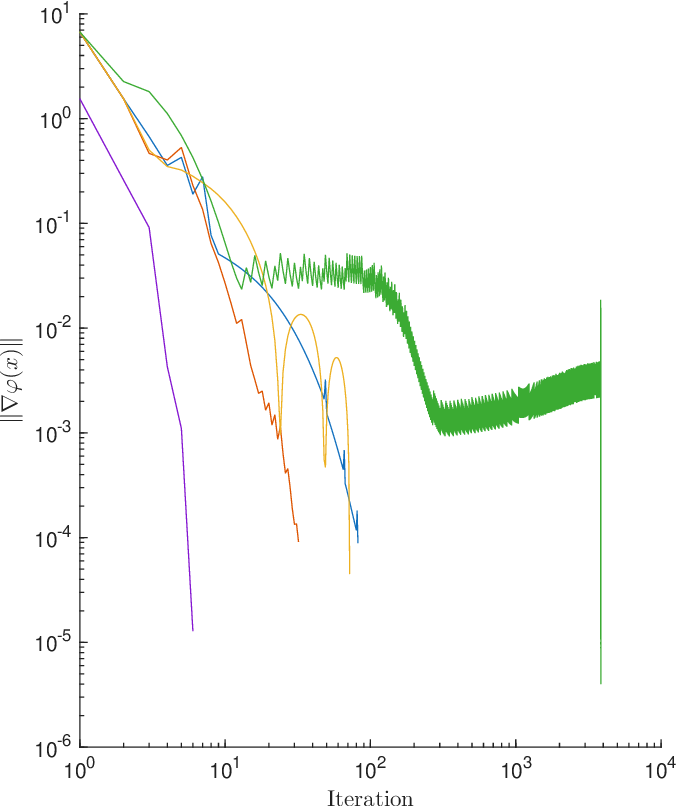}
    \end{subfigure}
    \begin{subfigure}{0.4\linewidth}
        \includegraphics[width=\linewidth]{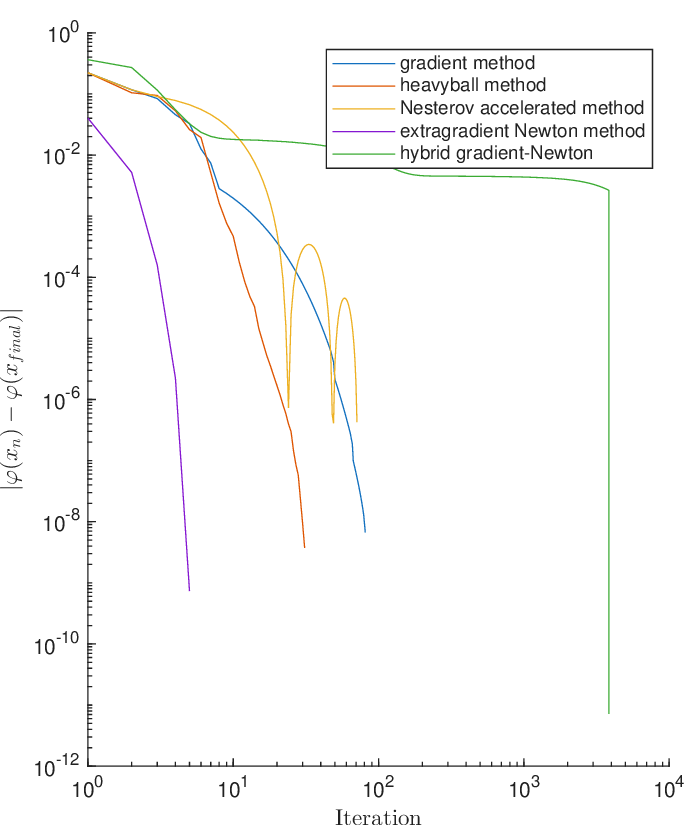}
    \end{subfigure}
    \caption{Convergence curves for Example \ref{example: box2}.}
    \label{fig: box 2}
\end{figure}

\end{Example}

\begin{Example} \rm \label{example: tstudent}

This problem is based on \cite{kmp} and is similar to the setting of \cite{Milzarek}. We define the objective 
\begin{equation} \label{eqn: students-t objective}
    \varphi(x) = \sum_{i=1}^N \log \left( 1+ \frac{(Ax-b)_i^2}{\nu} \right) + \mu \|x\|_2^2, \qquad x \in \R^N
\end{equation}
where $\nu, \mu >0$, the matrix $A$ and vector $b$ will be described shortly. Note that $\varphi$ is smooth but nonconvex. First we generate a sparse signal $\overline{x} \in \R^N$ with $k = [N/40]$ nonzero entries. The $k$ indices $i \in \{1,\hdots,N\}$ are randomly chosen with value 
\[
\overline{x}_i = \eta_1(i) 10^{\eta_2(i)},
\]
where $\eta_1(i) \in \{-1,1\}$ is a random sign and $\eta_2(i)$ is uniformly distributed in $[0,2].$ We define a matrix $A \in \R^{m \times N}$ that makes $m = N/8$ random cosine measurements, that is $Ax = (\text{dct}(x))_J$, where $J \subset \{ 1, \hdots, n \}$, $|J| =m$, is initialized randomly and dct is the discrete cosine transform.
Finally the input data $b \in \R^m$ is obtained by adding Student's-t noise with degree of freedom 4 rescaled by 0.1 to $A\overline{x}$. 

The gradient and Hessian of \eqref{eqn: students-t objective} are computed in \cite{kmp} and presented here for completeness. 
\begin{equation}
    \nabla \varphi (x)  = 2A^*u +2\mu x, \quad \text{where } u_i = \frac{(Ax-b)_i}{\nu + (Ax-b)^2_i }, 
\end{equation}
\begin{equation}
    \nabla^2 \varphi(x) = 2A^*\text{diag}(v) A + \mu I, \quad \text{where } v_i = \frac{\nu -(Ax-b)^2_i}{\left(\nu + (Ax-b)_i^2\right)^2}.
\end{equation}
As in \cite{Milzarek} the initial value is chosen as $A^*b$.

 We solve the problem for various sizes $N$ with results collected in Table \ref{table: tstudent}.  The extragradient Newton method is generally superior in terms of iterations compared to other tested methods, and specifically generally outperforms hybrid gradient-Newton. At worst, it is equivalent in iterations to hybrid gradient-Newton, which only occurs for the small scale of problems $N\leq32$. For this small scale hybrid gradient-Newton outperforms in runtime, as the cost of performing an extra gradient descent step is non-negligible at these scales. However, once the problems sizes scale up and the cost of computing a gradient becomes negligible compared to the Hessian, we see that extragradient Newton method outperforms hybrid gradient-Newton in both run-time and iterations. In particular for $N=256$, hybrid gradient-Newton performs unusually poorly, even more clearly seen in the convergence curves in Figure \ref{fig: tstudent 256 convergence}.
 
 In the middle regime $64 \leq N \leq 256$, where the costs to form the Hessian have yet to become expensive, our method outperforms all others significantly, which demonstrates the practicality of the algorithm on this problem, albeit on limited scales. 
\begin{table}[h]
\centering
\begingroup \scriptsize %
\begin {tabular}{c|cccccccccc}%
\toprule & \multicolumn {2}{c}{gradient method} & \multicolumn {2}{c}{heavyball method} &\multicolumn {2}{c}{\shortstack {Nesterov \\ accelerated method}} & \multicolumn {2}{c}{ \shortstack {extragradient\\ Newton method}} & \multicolumn {2}{c}{\shortstack {hybrid \\ gradient-Newton}}\\N&iterations&time&iterations&time&iterations&time&iterations&time&iterations&time\\\midrule %
\pgfutilensuremath {8}&\pgfutilensuremath {5}&\pgfutilensuremath {0.019}&\pgfutilensuremath {12}&\pgfutilensuremath {0.009}&\pgfutilensuremath {13}&\pgfutilensuremath {0.009}&\pgfutilensuremath {3}&\pgfutilensuremath {0.013}&\pgfutilensuremath {3}&\pgfutilensuremath {0.004}\\%
\pgfutilensuremath {16}&\pgfutilensuremath {6}&\pgfutilensuremath {0.006}&\pgfutilensuremath {15}&\pgfutilensuremath {0.005}&\pgfutilensuremath {8}&\pgfutilensuremath {0.003}&\pgfutilensuremath {2}&\pgfutilensuremath {0.003}&\pgfutilensuremath {4}&\pgfutilensuremath {0.001}\\%
\pgfutilensuremath {32}&\pgfutilensuremath {12}&\pgfutilensuremath {0.003}&\pgfutilensuremath {11}&\pgfutilensuremath {0.002}&\pgfutilensuremath {12}&\pgfutilensuremath {0.002}&\pgfutilensuremath {3}&\pgfutilensuremath {0.001}&\pgfutilensuremath {3}&\pgfutilensuremath {0.001}\\%
\pgfutilensuremath {64}&\pgfutilensuremath {116}&\pgfutilensuremath {0.022}&\pgfutilensuremath {115}&\pgfutilensuremath {0.022}&\pgfutilensuremath {292}&\pgfutilensuremath {0.056}&\pgfutilensuremath {11}&\pgfutilensuremath {0.003}&\pgfutilensuremath {19}&\pgfutilensuremath {0.004}\\%
\pgfutilensuremath {128}&\pgfutilensuremath {151}&\pgfutilensuremath {0.028}&\pgfutilensuremath {97}&\pgfutilensuremath {0.018}&\pgfutilensuremath {322}&\pgfutilensuremath {0.059}&\pgfutilensuremath {20}&\pgfutilensuremath {0.008}&\pgfutilensuremath {51}&\pgfutilensuremath {0.019}\\%
\pgfutilensuremath {256}&\pgfutilensuremath {370}&\pgfutilensuremath {0.070}&\pgfutilensuremath {305}&\pgfutilensuremath {0.058}&\pgfutilensuremath {422}&\pgfutilensuremath {0.076}&\pgfutilensuremath {36}&\pgfutilensuremath {0.020}&\pgfutilensuremath {385}&\pgfutilensuremath {0.168}\\%
\pgfutilensuremath {512}&\pgfutilensuremath {184}&\pgfutilensuremath {0.060}&\pgfutilensuremath {129}&\pgfutilensuremath {0.049}&\pgfutilensuremath {386}&\pgfutilensuremath {0.116}&\pgfutilensuremath {34}&\pgfutilensuremath {0.089}&\pgfutilensuremath {79}&\pgfutilensuremath {0.272}\\%
\pgfutilensuremath {1{,}024}&\pgfutilensuremath {259}&\pgfutilensuremath {0.109}&\pgfutilensuremath {157}&\pgfutilensuremath {0.064}&\pgfutilensuremath {883}&\pgfutilensuremath {0.493}&\pgfutilensuremath {76}&\pgfutilensuremath {0.362}&\pgfutilensuremath {157}&\pgfutilensuremath {0.692}\\%
\pgfutilensuremath {2{,}048}&\pgfutilensuremath {385}&\pgfutilensuremath {0.376}&\pgfutilensuremath {267}&\pgfutilensuremath {1.012}&\pgfutilensuremath {464}&\pgfutilensuremath {0.467}&\pgfutilensuremath {67}&\pgfutilensuremath {1.380}&\pgfutilensuremath {150}&\pgfutilensuremath {2.854}\\%
\end {tabular}%
\endgroup %

\caption{Performance comparison for Example \ref{example: tstudent}. Time is measured in seconds.}
\label{table: tstudent}
\end{table}

\begin{figure}
    \centering
    \begin{subfigure}{0.4\linewidth}
        \includegraphics[width=\linewidth]{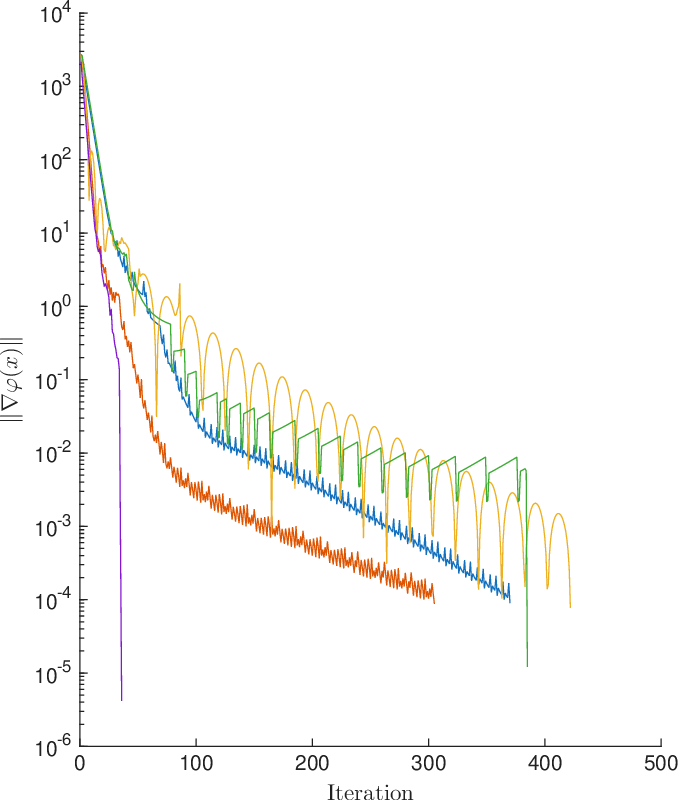} 
    \end{subfigure}
    \begin{subfigure}{0.4\linewidth}
        \includegraphics[width=\linewidth]{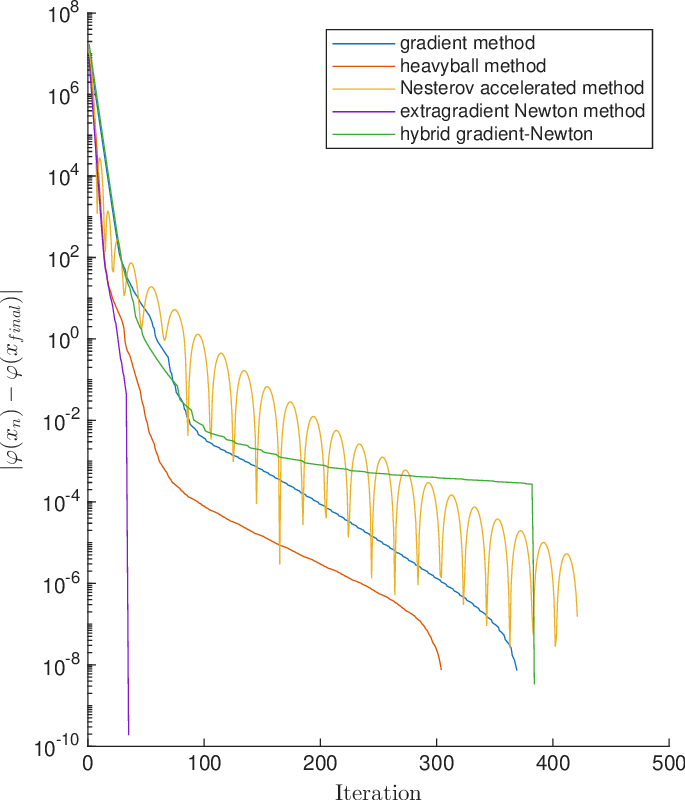}
    \end{subfigure}
    \caption{Convergence curves for Example \ref{example: tstudent} for $N = 256$.}
    \label{fig: tstudent 256 convergence}
\end{figure}
\end{Example}

\section{Concluding Remarks and Further Research}\label{sec:conclusion}

In this paper, we proposed a general line-search framework for Newton-type methods for unconstrained optimization. Unlike existing regularized Newton methods, the proposed framework exploits the Newton direction only when it is well defined and suitable, thereby avoiding Hessian regularization at every iteration. We established global convergence under mild assumptions, including the Polyak--\L ojasiewicz--Kurdyka (PLK) condition, together with local superlinear and quadratic convergence under appropriate regularity assumptions. We further demonstrated the applicability of the proposed framework to strongly quasiconvex optimization, providing a Newton-type algorithm with rigorous convergence guarantees for this important class of problems. The proposed framework naturally lead to a new extragradient Newton method that was shown effective  via numerical experiments.

\medskip 
Several directions for future research appear promising. A natural extension is to develop proximal Newton-type methods for structured nonsmooth and nonconvex optimization problems, where the objective function consists of smooth and nonsmooth components. Another important direction is to investigate stochastic variants of the proposed framework for large-scale optimization problems arising in machine learning and data science. It would also be of interest to extend the proposed methodology to constrained optimization and to explore adaptive strategies that further reduce the computational cost while preserving fast local convergence.


\section*{Funding and Conflicts of Interests} 

This research was partly supported by the Early Career Scholars
Program 2026, University of North Dakota under project 43700-2375-UND0031286. The authors declare that the presented results are new, and there is no any conflict of interest.
\end{document}